\documentclass[11pt,leqno]{article}
\usepackage{algorithm}
\usepackage{algpseudocode}
\usepackage{amsfonts}
\usepackage{amsmath}
\usepackage{amssymb} 
\usepackage{amsthm} 
\usepackage[title]{appendix} 
\usepackage{autobreak}
\usepackage{bigdelim} 
\usepackage{blkarray}
\makeatletter
\def\BA@fnsymbol#1{\ensuremath{
  \ifcase#1\or *\or \dagger\or \ddagger\or \mathsection\or \mathparagraph\or \|\or **\or \dagger\dagger\or \ddagger\ddagger \else@ctrerr\fi}}%
\makeatother
\usepackage{bm}
\usepackage{booktabs}
\usepackage{cancel} 
\usepackage{cases}
\usepackage{chngcntr}
\usepackage{colortbl} 
\usepackage{comment}
\usepackage{empheq} 
\usepackage{enumitem}
\setlist{leftmargin=*} 
\usepackage{fancyhdr} 
\usepackage{fancyvrb} 
\usepackage{framed} 
\usepackage{graphicx}
\usepackage{hyperref} 
\hypersetup{
  colorlinks,
  linkcolor={red!50!black},
  citecolor={blue!50!black},
  urlcolor={blue!80!black}
}
\usepackage[capitalize]{cleveref} 
\usepackage{mathrsfs} 
\usepackage{multirow} 
\usepackage[sort&compress,numbers]{natbib}
\usepackage{pifont}
\newcommand{\cmark}{\ding{51}}%
\newcommand{\xmark}{\ding{55}}%
\usepackage{setspace} 
\usepackage{stmaryrd}
\usepackage{subfig} 
\usepackage{pdflscape} 
\usepackage{pgfplots} 
\usepackage{pxrubrica} 
\usepackage{url}
\usepackage{tabularx}
\usepackage[colorinlistoftodos,prependcaption,textsize=footnotesize]{todonotes}

\usepackage[top=35truemm,bottom=35truemm,left=30truemm,right=30truemm]{geometry} 
\makeatletter

\@addtoreset{equation}{section}
\makeatother

\algtext*{EndWhile}
\algtext*{EndIf}
\algtext*{EndFor}

\theoremstyle{plain}
\newtheorem{theorem}{Theorem}[section]
\newtheorem{proposition}[theorem]{Proposition}%
\newtheorem{corollary}[theorem]{Corollary}
\newtheorem{lemma}[theorem]{Lemma}

\theoremstyle{definition}

\theoremstyle{remark}
\newtheorem{remark}[theorem]{Remark}%
\DeclareMathOperator*{\adj}{adj}

\DeclareMathOperator*{\setint}{int}
\DeclareMathOperator*{\setspan}{span}

\DeclareMathOperator*{\tr}{tr}

\DeclareMathOperator*{\Ker}{Ker}

\newcommand{\relmiddle}[1]{\mathrel{}\middle#1\mathrel{}}

\newcommand{\abs}[1]{\lvert #1 \rvert} 

\newcommand{\bbA}{\mathbb{A}}

\newcommand{\bbC}{\mathbb{C}}

\newcommand{\bbH}{\mathbb{H}}

\newcommand{\bbO}{\mathbb{O}}

\newcommand{\bbR}{\mathbb{R}}

\newcommand{\calD}{\mathcal{D}}

\newcommand{\calH}{\mathcal{H}}

\newcommand{\calS}{\mathcal{S}}

\newcommand{\CP}{\mathcal{CP}}
\newcommand{\COP}{\mathcal{COP}}

\makeatletter
\newcommand{\dotpplus}{\mathbin{\text{\@dotpplus}}}
\newcommand{\@dotpplus}{%
  \ooalign{\hidewidth\hbox{$\circ$}\hidewidth\cr$\m@th+$\cr}%
}
\makeatother

\makeatletter
\newcommand{\dotminus}{\mathbin{\text{\@dotminus}}}
\newcommand{\@dotminus}{%
  \ooalign{\hidewidth\hbox{$\circ$}\hidewidth\cr$\m@th-$\cr}%
}
\makeatother

\newcommand{\RNum}[1]{\uppercase\expandafter{\romannumeral #1\relax}} 
\newcommand{\Rnum}[1]{\lowercase\expandafter{\romannumeral #1\relax}} 

\title{Classification of facial exposedness of completely positive cones over symmetric cones}

\makeatletter
\let\@fnsymbol\@arabic
\makeatother

\author{
\normalsize
    Mitsuhiro Nishijima\thanks{Department of Industrial and Systems Engineering, Keio University, 3-14-1 Hiyoshi, Kohoku-ku, Yokohama-shi, 2238522, Kanagawa, Japan. ({\tt nishijima@keio.jp}).}
}

\begin{document}
\maketitle

\begin{abstract}\noindent
We classify the facial exposedness of completely positive cones over symmetric cones in terms of the rank of the associated Euclidean Jordan algebras.
The completely positive cones are facially exposed when the rank is at most $2$, but are not facially exposed when the rank is at least $5$.
Facial exposedness is not completely determined by the rank when the rank is $3$ or $4$, but we provide a characterization in each case.
The resulting classification of facial exposedness in fact agrees with the corresponding classification of spectrahedrality for completely positive cones.
\end{abstract}
\vspace{0.5cm}

\noindent
{\bf Key words. }completely positive cones, symmetric cones, facial structure, facial exposedness, spectrahedrality
%

\section{Introduction}
For a closed cone $K$ in a finite-dimensional real inner product space, we define
\begin{equation}
\CP(K) \coloneqq \left\{\sum_{i=1}^k a_i\otimes a_i \relmiddle| k\ge 1 \text{ and } a_i \in K \text{ for all $i = 1,\dots,k$}\right\},\label{eq:CPcone}
\end{equation}
and refer to it as the \emph{completely positive cone over $K$}.
A typical choice for $K$ in this paper is a \emph{symmetric cone}, namely, a cone that is self-dual and homogeneous.
Symmetric cones form a broad and fundamental class of convex cones, including nonnegative orthants, second-order cones, cones of real symmetric positive semidefinite matrices, and their finite direct products~\cite{FK1994}.
In the case where $K$ is a nonnegative orthant, the associated completely positive cone is essentially the cone of classical completely positive matrices, and a vast literature exists on this cone~\cite{SB2021}.
Motivated by the ability of completely positive cones to convexify various NP-hard problems~\cite{Bur2009,Bur2012,BMP2016,BD2012,BDd+2000}, researchers have studied the membership problem~\cite{Orl2021,HX20XX,FGN+2024}, approximations~\cite{NN2024_App,NN2024_Gen}, and the geometry~\cite{NL2025,NL2026,Nis2026} of completely positive cones and their duals, namely, copositive cones over symmetric cones.

This paper is devoted to investigating the \emph{facial exposedness}, a geometric property, of completely positive cones over symmetric cones.
A convex cone is called facially exposed if every face can be written as the intersection of the convex cone and its supporting hyperplane.
From the viewpoint of mathematical optimization, the notion of facial exposedness arose in connection with facial reduction algorithms~\cite{BW1981_Fac,BW1981_Reg}.
Facial exposedness is also important because certain geometric properties such as niceness~\cite{Pat2007,RT2019} and spectrahedrality~\cite{RG1995} require facial exposedness.
Spectrahedrality of a set, a notion that will play a role later in this paper, refers to the property that the set can be represented by a positive semidefinite constraint without extra variables.
See \cite{LRS2022} and \cite[section~1]{NL2025} for more detailed connections with other properties.

The facial exposedness of copositive cones over symmetric cones is completely settled.
By \cite[Theorem~3.3]{NL2025}, we see that a copositive cone is facially exposed if and only if the dimension of the underlying symmetric cone is at most $1$.

It is natural to ask whether the duals, completely positive cones over symmetric cones, are facially exposed.
When the underlying symmetric cone is a nonnegative orthant, the associated completely positive cone is facially exposed if and only if the dimension of the nonnegative orthant is at most $4$~\cite{Zha2018}; see also \cite{Kos2024,Kos2025,Zha2020} for the non-facial exposedness of completely positive cones over nonnegative orthants of dimension at least $5$.
When the underlying symmetric cones are a second-order cone and the direct product of the $1$-dimensional nonnegative orthant and a second-order cone, the associated completely positive cones are spectrahedral by Theorems~1 and 3 of \cite{SZ2003}, respectively, and hence they are facially exposed.
These results are independent of the dimension of the second-order cone, suggesting that one should look at a structural parameter of symmetric cones other than their dimension.

\begin{table}[tbp]
\caption{The facial exposedness of completely positive cones over symmetric cones in terms of the rank of the associated EJAs.
The symbol {\cmark} means that the completely positive cones are facially exposed without additional assumptions on the EJAs, and the symbol {\xmark} means that they are not.}
\label{tab:summary}
\begin{center}
\begin{tabular}{rl}
\toprule
Rank & Facial exposedness\\
\midrule
$\le 2$    & {\cmark} (\Cref{thm:rank1}) \\
$3$ & iff an EJA is not simple (\Cref{thm:rank3}) \\
$4$ & iff an EJA is Jordan-isomorphic to the Hadamard EJA $\bbR^4$ (\Cref{thm:rank4}) \\
$\ge 5$         & {\xmark} (\Cref{thm:rankge5})  \\
\bottomrule
\end{tabular}
\end{center}
\end{table}
In this paper, we classify the facial exposedness of completely positive cones over symmetric cones.
To do this, we realize symmetric cones as the cones of squares of \emph{Euclidean Jordan algebras} (EJAs) and investigate the facial exposedness of the associated completely positive cones in terms of the \emph{rank} of the EJAs.
Here, the rank of an EJA is a structural parameter whose precise definition is given in \cref{subsec:symcone}.
\Cref{tab:summary} summarizes the classification of the facial exposedness of completely positive cones over symmetric cones.
When the rank is at most $2$, the completely positive cones are facially exposed by existing results and simple arguments.
When the rank is at least $5$, the completely positive cones are not facially exposed because of the non-facial exposedness of completely positive cones over nonnegative orthants of dimension at least $5$.
Challenging situations arise when the rank is $3$ or $4$.
In these cases, additional case distinctions are essential to determine facial exposedness.

The classification of facial exposedness shown in Table~\ref{tab:summary} can also be read as that of spectrahedrality.
As can be seen from the subsequent discussion, if a completely positive cone over a symmetric cone is facially exposed, it is actually spectrahedral.
Note that \cite[Theorem~4.1]{Ros2014} provides an example of a closed convex cone that is facially exposed but not spectrahedral; these properties can be verified using the proof of \cite[Proposition~5]{Pat2013_On}, or alternatively Proposition~2.3 and Corollary~3.5 of \cite{LRS2022}.
Thus, facial exposedness does not imply spectrahedrality in general.
The classification of spectrahedrality of completely positive cones over symmetric cones is related to the question of whether the completely positive cones and their duals are \emph{spectrahedral shadows}, namely projections of spectrahedra~\cite[Question~6.1]{Nis2026}.
This classification problem for spectrahedral shadows in the rank-$3$ and rank-$4$ cases remains open.

The organization of this paper is as follows.
\Cref{sec:preliminary} collects the notation and technical results needed throughout the paper.
\Cref{sec:rankle2,sec:rank3,sec:rank4,sec:rank5} discuss the facial exposedness of completely positive cones over symmetric cones when the ranks of the associated EJAs are at most $2$, equal to $3$, equal to $4$, and at least $5$, respectively.

\section{Preliminaries}\label{sec:preliminary}
\subsection{Basic notation}
Let $\bbR^n$ denote the space of $n$-dimensional real vectors equipped with the standard inner product.
Vectors in $\bbR^n$ are column vectors denoted by boldface lowercase letters, such as $\bm{a}$.
Let $\bm{a}\in \bbR^n$.
Vectors are indexed from $1$, and $a_i$ denotes the $i$th entry of $\bm{a}$.
For $1\le i < j \le n$, we write $\bm{a}_{i:j}$ for the vector obtained by extracting the $i$th to $j$th entries of $\bm{a}$.
The transpose of $\bm{a}$ is denoted by $\bm{a}^\top$.
We use $\bm{0}$ and $\bm{e}_i$ to denote the vector all of whose entries are $0$ and the vector whose $i$th entry is $1$ and the others are $0$, respectively.

Matrices are denoted by boldface uppercase letters, such as $\bm{A}$.
Rows and columns of matrices are indexed from $1$, and $A_{ij}$ denotes the $(i,j)$th entry of a matrix $\bm{A}$.
We use $\bm{O}$, $\bm{I}_n$, and $\bm{J}_n$ to denote the matrix all of whose entries are $0$, the identity matrix of order $n$, and the diagonal matrix of order $n$ whose $(1,1)$th entry is $1$ and the other diagonal entries are $-1$, respectively.
For a square matrix $\bm{A}$, we use $\Ker(\bm{A})$, $\det(\bm{A})$, $\tr(\bm{A})$, $\adj(\bm{A})$, and $\{0\}\oplus \bm{A}$ to denote the kernel of $\bm{A}$, the determinant of $\bm{A}$, the trace of $\bm{A}$, the adjugate of $\bm{A}$, and the square matrix whose entries in the first row and column are $0$ and whose principal submatrix obtained by deleting the first row and column is $\bm{A}$, respectively.
We also use the symbol $\oplus$ to denote the direct sum of vector spaces.
We use $\calS^n$ and $\calS_+^n$ to denote the space of real $n\times n$ symmetric matrices and the cone of positive semidefinite matrices in $\calS^n$, respectively.
For a real symmetric positive definite matrix $\bm{A}$, we define $\sqrt{\bm{A}}$ to be the unique real symmetric positive definite matrix satisfying $(\sqrt{\bm{A}})^2 = \bm{A}$~\cite[page~440]{HJ2013}.

Let $V$ be a finite-dimensional real vector space and $S$ be a set in $V$.
We use $\setspan S$ and $\setint S$ to denote the smallest subspace containing $S$ and the interior of $S$, respectively.
In addition, we define $\bbR S \coloneqq \{ax \mid a\in \bbR,\ x\in S\}$ and $\bbR_+ S \coloneqq \{ax \mid a\ge 0,\ x\in S\}$.
For $x\in V$, we write $\bbR \{x\}$ and $\bbR_+\{x\}$ as $\bbR x$ and $\bbR_+x$, respectively.

We further assume that $V$ is equipped with an inner product denoted by $\bullet$.
For a set $S$ in $V$, we write $S^\perp$ for the space of $x\in V$ such that $x\bullet y = 0$ for all $y\in S$.
For a set $K$ in $V$, we call $K$ a \emph{cone} if $ax\in K$ for all $a\ge 0$ and $x\in K$, and we define $K^*$, the dual of $K$, as the cone of $x\in V$ such that $x\bullet y \ge 0$ for all $y\in K$.
For two closed cones $K_1$ and $K_2$, we say that $K_1$ is \emph{linearly isomorphic} to $K_2$ if there exists a linear isomorphism $\phi\colon \setspan K_1 \to \setspan K_2$ such that $\phi(K_1) = K_2$.
For $x\in V$, we write $x\otimes x$ for the tensor product of $x$ with itself, which can be regarded as the linear transformation on $V$ that maps $y$ to $(x\bullet y)x$.
In this paper, the notation $\bm{X} \otimes \bm{X}$ is used for a matrix $\bm{X}$, but this does not denote the Kronecker product.
The space of self-adjoint linear transformations on $V$ is denoted by $\calS(V)$, and $\langle \cdot,\cdot\rangle$ denotes the inner product on $\calS(V)$ induced by that on $V$.
When $V = \bbR^n$, we identify $\calS(\bbR^n)$ with $\calS^n$, where the inner product $\langle \cdot,\cdot\rangle$ on $\calS^n$ is given by $\langle \bm{X},\bm{Y}\rangle = \sum_{i,j=1}^n X_{ij}Y_{ij}$ for $\bm{X},\bm{Y}\in \calS^n$.
We define $H_m(V)$ as the space of homogeneous polynomials of degree $m$ with real coefficients on $V$.
We note that elements in $H_m(\bbR^n)$ are homogeneous polynomials in the usual sense.
For $A\in \calS(V)$, we define $q_A\in H_2(V)$ as
\begin{equation}
q_A(x) \coloneqq x\bullet A(x) \label{eq:def_qA}
\end{equation}
for $x\in V$.

For functions $f_1,\dots,f_k$, we write $f_1\circ \dots \circ f_k$ for the composite function defined recursively as $(f_1\circ \dots \circ f_k)(x) \coloneqq (f_1\circ \dots \circ f_{k-1})(f_k(x))$.
We also use the symbol $\circ$ to denote the product of EJAs defined in \cref{subsec:symcone}.
For a linear mapping $f$ from a finite-dimensional real inner product space to another, $f^*$ denotes its adjoint. 

\subsection{Faces and their exposedness}
Let $K$ be a closed convex cone in a finite-dimensional real inner product space $V$.
A \emph{face} of $K$ is a nonempty (necessarily closed) convex subcone $F$ such that for any $a,b\in K$, we have $a,b\in F$ if $a + b \in F$.
We say that a nonzero $x\in K$ generates an \emph{extreme ray} of $K$ if $\bbR_+ x$ is a face of $K$.
An \emph{exposed face} of $K$ is a set that can be written as the intersection of $K$ and $\{d\}^\perp$ for some $d \in K^*$.
An exposed face is necessarily a face as the name suggests.
We call $K$ \emph{facially exposed} if all the faces of $K$ are exposed.

One way to show that a closed convex cone is not facially exposed is to exhibit one of its faces that is not facially exposed as a cone.
This argument is justified by the following theorem, which follows immediately from the discussion in \cite[page~163]{Roc1970}.

\begin{theorem}\label{thm:nonexposed}
Let $K$ be a closed convex cone in $V$, let $F_1$ be a face of $K$, and let $F_2$ be a nonempty convex subcone of $F_1$.
\begin{enumerate}[label=(\roman*), ref=\roman*]
\item $F_2$ is a face of $K$ if and only if $F_2$ is a face of $F_1$. \label{enum:face_iff}
\item $F_2$ is a non-exposed face of $K$ if $F_2$ is a non-exposed face of $F_1$.
\end{enumerate}
In particular, $K$ is not facially exposed if $F_1$ is not facially exposed as a cone.
\end{theorem}

In addition, we investigate the facial exposedness of a closed convex cone via another cone that is linearly isomorphic to the cone under consideration.
It is well known that facial exposedness is invariant under linear isomorphism.
We summarize this in the following theorem and provide its concise proof because it is crucial for the discussion in this paper.

\begin{theorem}\label{thm:isom_face_exposed}
For each $i \in \{1,2\}$, let $K_i$ be a closed convex cone in a finite-dimensional inner product space.
Assume that $K_1$ is linearly isomorphic to $K_2$, and let $\phi\colon \setspan K_1 \to \setspan K_2$ be a linear isomorphism satisfying $\phi(K_1) = K_2$.
Let $F$ be a nonempty convex subcone of $K_1$.
\begin{enumerate}[label=(\roman*), ref=\roman*]
\item $F$ is a face of $K_1$ if and only if $\phi(F)$ is a face of $K_2$. \label{enum:iff_face}
\item $F$ is an exposed face of $K_1$ if and only if $\phi(F)$ is an exposed face of $K_2$. \label{enum:iff_exposed}
\end{enumerate}
In particular, $K_1$ is facially exposed if and only if so is $K_2$.
\end{theorem}

\begin{proof}
We can verify \eqref{enum:iff_face} from the definition of faces.
We can deduce \eqref{enum:iff_exposed} from the fact that if $F = K_1 \cap \{d\}^\perp$ for some $d\in K_1^*$, then we have $(\phi^*)^{-1}(d) \in K_2^*$ and $\phi(F) = K_2 \cap \{(\phi^*)^{-1}(d)\}^\perp$.
\end{proof}

\emph{Spectrahedrality} of a closed convex cone is a sufficient condition for the facial exposedness of the cone~\cite[Corollary~1]{RG1995}.
A closed convex cone $K$ in $V$ is \emph{spectrahedral} if there exist $d\ge 1$ and an affine mapping $\bm{A}\colon V\to \calS^d$ such that $K = \{x\in V \mid \bm{A}(x) \in \calS_+^d\}$.
For later use, we summarize this fact as the following theorem.

\begin{theorem}\label{thm:spec_implies_exposed}
If $K$ is a spectrahedral cone in $V$, then it is facially exposed.
\end{theorem}

\subsection{Euclidean Jordan algebras and symmetric cones}\label{subsec:symcone}
A \emph{Jordan algebra} is a pair $(E,\circ)$ where $E$ is a finite-dimensional real vector space and $\circ\colon E \times E \to E$ is a bilinear product that is commutative and satisfies $x\circ((x\circ x)\circ y) = (x\circ x)\circ (x\circ y)$ for all $x,y \in E$.
A Jordan algebra $(E,\circ)$ with a unit element is said to be \emph{Euclidean} if it admits an inner product $\bullet \colon E \times E \to \bbR$ that is associative in the sense that $(x\circ y)\bullet z = x\bullet (y\circ z)$ for all $x,y,z\in E$.
In this paper, we may write an EJA as $E$ for simplicity or as $(E,\circ,\bullet)$ to specify an underlying associative inner product on $E$.
For two EJAs $(E_1,\circ_1)$ and $(E_2,\circ_2)$, we say that $E_1$ is \emph{Jordan-isomorphic} to $E_2$ if there exists a linear isomorphism $\phi\colon E_1 \to E_2$ such that $\phi(x\circ_1 y) = \phi(x) \circ_2 \phi(y)$ for all $x,y\in E_1$, and we call such a mapping $\phi$ a \emph{Jordan isomorphism}.

Let $(E,\circ)$ be an EJA.
An element $c \in E$ is called an \emph{idempotent} if $c\circ c = c$.
For $c_1,\dots,c_r\in E$, we call $\{c_1,\dots,c_r\}$ a \emph{Jordan frame} of $E$ if every $c_i$ for $i\in \{1,\dots,r\}$ is a nonzero idempotent that cannot be written as the sum of two nonzero idempotents in $E$, the sum of $c_1,\dots,c_r$ is a unit element of $E$, and $i\neq j$ implies $c_i\circ c_j = 0$.
The cardinality $r$ of a Jordan frame does not depend on the choice of the Jordan frame; this integer is called the \emph{rank} of $E$.

For every EJA $E$, the set $E_+ \coloneqq \{x\circ x \mid x\in E\}$ forms a \emph{symmetric cone}, namely, a self-dual and homogeneous cone~\cite[Theorem~\RNum{3}.2.1]{FK1994}.
Conversely, any symmetric cone can be represented as $E_+$ for some EJA $E$, and the EJA is unique up to Jordan isomorphism~\cite[Chapter~\RNum{3}]{FK1994}.
By the invariance of the rank of EJAs under Jordan isomorphism, we see that for a given symmetric cone, the algebraically defined rank of a corresponding EJA is uniquely determined.
Indeed, the rank is completely determined by the geometry of the symmetric cone, and is computed as the maximum number of mutually orthogonal extreme rays of the symmetric cone~\cite[section~2.2]{Nis2026}.
Therefore, it makes sense to classify the facial exposedness of completely positive cones over symmetric cones in terms of the rank of the associated EJAs.

Here we display some examples of EJAs.
First, let $E \coloneqq \bbR^n$ with $n\ge 1$ and define
\begin{equation*}
\bm{x} \circ \bm{y} \coloneqq (x_1y_1,\dots,x_ny_n)^\top
\quad\text{and}\quad
\bm{x} \bullet \bm{y} \coloneqq \bm{x}^\top\bm{y},
\end{equation*}
respectively, for $\bm{x},\bm{y} \in \bbR^n$.
Then $(E,\circ,\bullet)$ forms an EJA, and we call it a \emph{Hadamard EJA}.
The rank of this EJA is $n$, and the associated symmetric cone is the nonnegative orthant $\bbR_+^n$, i.e., the cone of nonnegative vectors in $\bbR^n$.
Second, let $E \coloneqq \bbR^n$ with $n\ge 2$ and define
\begin{equation*}
\bm{x} \circ \bm{y} \coloneqq \begin{pmatrix}
\bm{x}^\top\bm{y}\\
x_1\bm{y}_{2:n} + y_1\bm{x}_{2:n}
\end{pmatrix}
\quad\text{and}\quad
\bm{x} \bullet \bm{y} \coloneqq \bm{x}^\top\bm{y},
\end{equation*}
respectively, for $\bm{x},\bm{y} \in \bbR^n$.
We call the EJA defined in this way a \emph{Jordan spin algebra}, and it is denoted by $L^n$.
The rank of this EJA is $2$, and the associated symmetric cone is the second-order cone defined as
\begin{equation*}
L_+^n \coloneqq \left\{\bm{x}\in \bbR^n \relmiddle| x_1 \ge \sqrt{x_2^2 + \cdots + x_n^2}\right\}.
\end{equation*}
We note that $L^n$ itself may be used to denote the second-order cone~\cite{BN2001,NL2026}, but in this paper, following \cite{GST2004,Orl2025_JorAut} for example, this symbol is used to represent a specific EJA associated with the cone, i.e., the Jordan spin algebra.
Third, let $\bbC$, $\bbH$, and $\bbO$ denote the sets of complex numbers, quaternions, and octonions, respectively.
For $\bbA \in \{\bbR,\bbC,\bbH,\bbO\}$ and $n\ge 1$, we define $\calH^n(\bbA)$ as the space of Hermitian $n\times n$ matrices with entries in $\bbA$.
In addition, we define
\begin{equation*}
\bm{X} \circ \bm{Y} \coloneqq \frac{\bm{X}\bm{Y} + \bm{Y}\bm{X}}{2} 
\quad\text{and}\quad
\bm{X} \bullet \bm{Y} \coloneqq \tr(\bm{X}\circ \bm{Y}),
\end{equation*}
respectively, for $\bm{X},\bm{Y} \in \calH^n(\bbA)$.
Then $(\calH^n(\bbA),\circ,\bullet)$ forms an EJA for any $n$ if $\bbA \in \{\bbR,\bbC,\bbH\}$ or for $n\le 3$ if $\bbA = \bbO$, and this EJA has rank $n$~\cite[Chapter~\RNum{5}]{FK1994}.
In particular, $\calH^n(\bbR) = \calS^n$ and the associated symmetric cone is $\calS_+^n$.

Again, let $(E,\circ)$ be an EJA.
We refer to $E'$ as a \emph{subalgebra} of $E$ if it is a subspace of $E$ and closed under the product $\circ$.
For every subalgebra $E'$ of $E$, the orthogonal projection (with respect to any fixed associative inner product) of the unit element of $E$ onto $E'$ is a unit element of $E'$, and hence $E'$ is an EJA.
Two subalgebras $E_1$ and $E_2$ of $E$ are called \emph{orthogonal} if $x \circ y  = 0$ for all  $x\in E_1$ and $y\in E_2$.
The EJA $E$ is called \emph{simple} if it is nonzero and cannot be written as the direct sum of two nonzero orthogonal subalgebras.
By \cite[Proposition~\RNum{3}.4.4]{FK1994}, every nonzero EJA $E$ can be decomposed into the direct sum of simple subalgebras $E_1,\dots,E_k$ that are pairwise orthogonal.
We note that the rank of $E$ is the sum of the ranks of $E_1,\dots,E_k$ since the union of a Jordan frame of each $E_i$ for $i \in \{1,\dots,k\}$ forms a Jordan frame of $E$.
The classification of simple EJAs is well known~\cite[Chapter~\RNum{5}]{FK1994}.
The following theorem summarizes the relevant classification results for EJAs of ranks $1$, $2$, and $3$ in a form convenient for our purposes.

\begin{theorem}\label{thm:EJA_isom}
Let $E$ be a simple EJA.
\begin{enumerate}[label=(\roman*), ref=\roman*]
\item If $E$ has rank $1$, then it is Jordan-isomorphic to the Hadamard EJA $\bbR$. \label{enum:EJA_isom_rank1}
\item If $E$ has rank $2$, then there exists $n\ge 3$ such that the algebra is Jordan-isomorphic to the Jordan spin algebra $L^n$. \label{enum:EJA_isom_rank2}
\item If $E$ has rank $3$, then it is Jordan-isomorphic to $\calH^3(\bbA)$ for some $\bbA \in \{\bbR,\bbC,\bbH,\bbO\}$.\label{enum:EJA_isom_rank3_simple}
\end{enumerate}
\end{theorem}
\begin{proof}
If $E$ has rank $1$, we have $E = \bbR e$, where $e$ is the unit element of $E$.
The linear mapping sending $e$ to $1$ is a Jordan isomorphism between $E$ and the Hadamard EJA $\bbR$.
Therefore, we obtain \eqref{enum:EJA_isom_rank1}.
\eqref{enum:EJA_isom_rank2} is a consequence of \cite[Corollary~\RNum{4}.1.5]{FK1994}.
\eqref{enum:EJA_isom_rank3_simple} follows directly from \cite[Theorem~\RNum{5}.3.7]{FK1994}.
\end{proof}

\subsection{Completely positive and copositive cones}
Let $K$ be a closed cone in a finite-dimensional real inner product space $V$.
Recall the definition of $\CP(K)$ provided in \eqref{eq:CPcone}.
Its dual is the \emph{copositive cone over $K$} defined as
\begin{equation*}
\COP(K) \coloneqq \{A \in \calS(V) \mid  q_A(x) \ge 0 \text{ for all $x\in K$}\},
\end{equation*}
where $q_A$ is defined in \eqref{eq:def_qA}.
See \cite[section~2]{SZ2003} for their duality.

First, we observe that a linear isomorphism between two underlying cones induces a linear isomorphism between the associated completely positive cones and preserves facial exposedness.

\begin{lemma}\label{lem:CP_isom}
For each $i \in \{1,2\}$, let $K_i$ be a closed cone in a finite-dimensional real inner product space.
We assume that $K_1$ is linearly isomorphic to $K_2$.
Then $\CP(K_1)$ is linearly isomorphic to $\CP(K_2)$.
In particular, $\CP(K_1)$ is facially exposed if and only if so is $\CP(K_2)$.
\end{lemma}

\begin{proof}
Let $\phi\colon\setspan K_1\to \setspan K_2$ be a linear isomorphism satisfying $\phi(K_1) = K_2$.
For any $A\in \CP(K_1)$, we have
\begin{equation}
\phi \circ A \circ \phi^* \in \CP(K_2) \label{eq:PhiA}
\end{equation}
since for any $a\in K_1$, it follows that $\phi\circ (a\otimes a)\circ \phi^* = \phi(a)\otimes \phi(a)$ and $\phi(a) \in K_2$.

We define $\Phi\colon \setspan\CP(K_1) \to \setspan\CP(K_2)$ as $\Phi(A) \coloneqq \phi \circ A \circ \phi^*$ for every $A\in \setspan\CP(K_1)$.
Since every element in $\setspan\CP(K_1)$ can be expressed as the difference of two elements in $\CP(K_1)$~\cite[Theorem~2.7]{Roc1970}, it follows from \eqref{eq:PhiA} that $\Phi$ is well-defined.
The inverse of $\Phi$ is given by $\Phi^{-1}(B) = \phi^{-1}\circ B \circ (\phi^*)^{-1}$ for $B\in \setspan\CP(K_2)$, and thus $\Phi$ is a linear isomorphism.
In addition, using \eqref{eq:PhiA}, we obtain $\Phi(\CP(K_1)) = \CP(K_2)$.
Therefore, $\CP(K_1)$ is linearly isomorphic to $\CP(K_2)$.
Finally, by \cref{thm:isom_face_exposed}, we obtain the equivalence of facial exposedness for the two cones.
\end{proof}

In particular, we see from \cref{lem:CP_isom} that for an EJA $E$, the facial exposedness of $\CP(E_+)$ is independent of the choice of an associative inner product on $E$.
This means that it is not necessary to specify the associative inner product, and at the same time, we can choose it arbitrarily to investigate facial exposedness in practical computations.
More generally, the following corollary states the invariance of the facial exposedness of $\CP(E_+)$ under Jordan isomorphism.

\begin{corollary}\label{cor:CP_isom_Jordan}
For each $i\in \{1,2\}$, let $(E_i,\circ_i)$ be an EJA.
We assume that $E_1$ is Jordan-isomorphic to $E_2$.
Then $\CP((E_1)_+)$ is linearly isomorphic to $\CP((E_2)_+)$.
In particular, $\CP((E_1)_+)$ is facially exposed if and only if so is $\CP((E_2)_+)$.
\end{corollary}

\begin{proof}
By assumption, the cone $(E_1)_+$ is linearly isomorphic to $(E_2)_+$.
Therefore, from \cref{lem:CP_isom}, we obtain the desired result.
\end{proof}

Second, we summarize results on exposed faces of completely positive cones.
The following result is quoted from \cite[Proposition~5.4]{NL2026} for later use.

\begin{theorem}\label{thm:CP_face_subsp}
Let $K$ be a closed cone in $V$ and $X$ be a subspace of $V$.
Then $\CP(K \cap X)$ is an exposed face of $\CP(K)$.
\end{theorem}

\Cref{thm:CP_face_subsp} yields a class of exposed faces of completely positive cones over symmetric cones as shown in the following lemma.

\begin{lemma}\label{lem:subalg_expface}
Let $E$ be an EJA and $E'$ be a subalgebra of $E$.
Then $\CP(E_+')$ is an exposed face of $\CP(E_+)$.
\end{lemma}

\begin{proof}
We note that $E_+ \cap E' = E_+'$.
We can derive this equality from \cite[equation~(9)]{NN2016}, or alternatively prove it using spectral decomposition~\cite[Theorem~\RNum{3}.1.2]{FK1994}.
Since $E_+'$ can be written as the intersection of the cone $E_+$ and the subspace $E'$, by \cref{thm:CP_face_subsp}, we obtain the desired result.
\end{proof}

The following lemma provides us with a way to describe exposed faces of completely positive cones in terms of zeros of quadratic forms.
Its proof is omitted because it follows directly from the definitions.
We note that Dickinson~\cite[page~393]{Dic2011} pointed out the same result in the case where $K$ is a nonnegative orthant.

\begin{lemma}\label{lem:CP_exposed_face}
Let $K$ be a closed cone in $V$ and $D\in \COP(K)$.
We define
\begin{equation}
Z_{K}(D) \coloneqq \{x\in K \mid q_D(x) = 0\}. \label{eq:def_ZKD}
\end{equation}
Then we have $\CP(K) \cap \{D\}^\perp = \CP(Z_{K}(D))$.
\end{lemma}

For some specific symmetric cones, the facial exposedness (spectrahedrality) of the associated completely positive cones has been determined.
\begin{theorem}\label{thm:CPn_facially_exposed}
Let $n \ge 1$.
Then $\CP(\bbR_+^n)$ is facially exposed (spectrahedral) if and only if $n\le 4$.
\end{theorem}

\begin{proof}
If $n\le 4$, it follows from \cite{MM1962} that $\CP(\bbR_+^n)$ is equal to the intersection of $\calS_+^n$ and the cone of entrywise nonnegative matrices in $\calS^n$.
This implies that $\CP(\bbR_+^n)$ is spectrahedral, and therefore facially exposed by \cref{thm:spec_implies_exposed}.
If $n\ge 5$, it follows from \cite{Zha2018} that $\CP(\bbR_+^n)$ is not facially exposed and thus not spectrahedral either.
\end{proof}
\begin{theorem}[{\cite[Theorem~1]{SZ2003}}]\label{thm:CPLn}
Let $n\ge 2$.
Then we have
\begin{equation*}
\CP(L_+^n) = \{\bm{X}\in \calS_+^n \mid \langle\bm{J}_n,\bm{X}\rangle \ge 0\}. \label{eq:CPLn}
\end{equation*}
In particular, $\CP(L_+^n)$ is facially exposed.
\end{theorem}

\begin{theorem}[{\cite[Theorem~3]{SZ2003}}]\label{thm:CP_nno1soc}
Let $n\ge 2$.
Then we have
\begin{equation*}
\CP(\bbR_+\times L_+^n) = \{\bm{X} \in \calS_+^{n+1} \mid \langle \{0\}\oplus \bm{J}_n,\bm{X}\rangle \ge 0,\ (\bm{X}\bm{e}_1)_{2:n+1} \in L_+^n\}. \label{eq:rank3_12_rhs}
\end{equation*}
In particular, $\CP(\bbR_+\times L_+^n)$ is facially exposed.
\end{theorem}

\begin{remark}
Burer and Dong~\cite{BD2012} defined
\begin{equation*}
\calD \coloneqq  \{\bm{X}\in \calS_+^{n+1} \mid \bm{X}\bm{s} \in \bbR_+\times L_+^n \text{ for all $\bm{s}$ generating extreme rays of $\bbR_+\times L_+^n$}\}
\end{equation*}
and asserted that $\calD$ is tractable (in the sense that the separation problem for $\calD$ can be solved in polynomial time) and equal to $\CP(\bbR_+\times L_+^n)$.
In fact, $\calD$ is \emph{not} equal to $\CP(\bbR_+\times L_+^n)$ for $n \ge 3$.
To see this, let $\bm{X} \coloneqq \{0\}\oplus \bm{I}_n$.
Since the extreme rays of $\bbR_+\times L_+^n$ are $\bbR_+\bm{e}_1$ and $\bbR_+(0,1,\bm{v}^\top)^\top$ with $\bm{v} \in \bbR^{n-1}$ of norm $1$, we can see that $\bm{X} \in \calD$.
On the other hand, since $\langle \{0\}\oplus \bm{J}_n,\bm{X}\rangle = -n +2 < 0$, it follows from \cref{thm:CP_nno1soc} that $\bm{X} \not\in \CP(\bbR_+\times L_+^n)$.
\end{remark}

\section{Case of rank at most $2$}\label{sec:rankle2}
In this section, we show that $\CP(E_+)$ is facially exposed for any EJA $E$ of rank at most $2$.

\begin{theorem}\label{thm:rank1}
Let $E$ be an EJA of rank at most $2$.
Then $\CP(E_+)$ is facially exposed.
\end{theorem}

\begin{proof}
If $E$ has rank $0$, i.e., $E = \{0\}$, then we have $\CP(E_+) = \{0\}$, which is facially exposed.
If $E$ has rank $1$, then we see from \eqref{enum:EJA_isom_rank1} of \cref{thm:EJA_isom} that $E$ is Jordan-isomorphic to the Hadamard EJA $\bbR$.
Since $\CP(\bbR_+)$ is facially exposed by \cref{thm:CPn_facially_exposed}, we see from \cref{cor:CP_isom_Jordan} that $\CP(E_+)$ is facially exposed.

In what follows, we consider the case where $E$ has rank $2$.
Let $E_1,\dots,E_k$ be simple subalgebras that are pairwise orthogonal and whose direct sum is $E$, and let $r_i \ge 1$ denote the rank of $E_i$ for every $i \in \{1,\dots,k\}$.
There are the following two possibilities for $(r_1,\dots,r_k)$ satisfying $\sum_{i=1}^k r_i = 2$:
\begin{enumerate}[label=Case~2.\alph*, ref=2.\alph*]
\item $k = 2$ and $(r_1,r_2) = (1,1)$, \label{enum:rank2_11}
\item $k = 1$ and $r_1 = 2$. \label{enum:rank2_2}
\end{enumerate}

First, we consider Case~\ref{enum:rank2_11}.
In this case, by \eqref{enum:EJA_isom_rank1} of \cref{thm:EJA_isom}, we see that $E$ is Jordan-isomorphic to the Hadamard EJA $\bbR^2$.
Since $\CP(\bbR_+^2)$ is facially exposed by \cref{thm:CPn_facially_exposed}, we see from \cref{cor:CP_isom_Jordan} that $\CP(E_+)$ is facially exposed.

Second, we consider Case~\ref{enum:rank2_2}.
By \eqref{enum:EJA_isom_rank2} of \cref{thm:EJA_isom}, $E$ is Jordan-isomorphic to a Jordan spin algebra $L^n$ for some $n \ge 3$.
Since $\CP(L_+^n)$ is facially exposed by \cref{thm:CPLn}, we see from \cref{cor:CP_isom_Jordan} that $\CP(E_+)$ is facially exposed.
\end{proof}

\section{Case of rank $3$}\label{sec:rank3}
Let $E$ be an EJA of rank $3$.
In addition, let $E_1,\dots,E_k$ be simple subalgebras that are pairwise orthogonal and whose direct sum is $E$, and let $r_i \ge 1$ denote the rank of $E_i$ for every $i \in \{1,\dots,k\}$.
Without loss of generality, we may assume that $r_1 \le  \cdots \le r_k$.
There are the following three possibilities for $(r_1,\dots,r_k)$ satisfying $\sum_{i=1}^k r_i = 3$:
\begin{enumerate}[label=Case~3.\alph*, ref=3.\alph*]
\item $k = 3$ and $(r_1,r_2,r_3) = (1,1,1)$, \label{enum:rank3_111}
\item $k = 2$ and $(r_1,r_2) = (1,2)$, \label{enum:rank3_12}
\item $k = 1$ and $r_1 = 3$. \label{enum:rank3_3}
\end{enumerate}

In Case~\ref{enum:rank3_111}, the facial exposedness can be shown in a way similar to that in the proof of \cref{thm:rank1}, so we state the result and omit its proof.
\begin{proposition}\label{prop:rank3_111}
Under the assumption of Case~\ref{enum:rank3_111}, $\CP(E_+)$ is facially exposed.
\end{proposition}

In Case~\ref{enum:rank3_12}, $\CP(E_+)$ is facially exposed as shown in the following proposition.

\begin{proposition}\label{prop:rank3_12}
Under the assumption of Case~\ref{enum:rank3_12}, $\CP(E_+)$ is facially exposed.
\end{proposition}

\begin{proof}
By \eqref{enum:EJA_isom_rank1} and \eqref{enum:EJA_isom_rank2} of \cref{thm:EJA_isom}, $E$ is Jordan-isomorphic to the direct product of the Hadamard EJA $\bbR$ and a Jordan spin algebra $L^n$ for some $n\ge 3$.
Since $\CP(\bbR_+ \times L_+^n)$ is facially exposed by \cref{thm:CP_nno1soc}, we see from \cref{cor:CP_isom_Jordan} that $\CP(E_+)$ is facially exposed.
\end{proof}

In Case~\ref{enum:rank3_3}, i.e., the case where $E$ is a simple EJA of rank $3$, we will see that $\CP(E_+)$ is not facially exposed.
We achieve this by finding a face of $\CP(E_+)$ that is linearly isomorphic to $\CP(\calS_+^3)$ in \cref{lem:rank3_3_isom}, and then proving the non-facial exposedness of $\CP(\calS_+^3)$ in \cref{lem:CP_PSD3}.

\begin{lemma}\label{lem:rank3_3_isom}
Under the assumption of Case~\ref{enum:rank3_3}, $\CP(E_+)$ has an exposed face that is linearly isomorphic to $\CP(\calS_+^3)$.
\end{lemma}

\begin{proof}
By \eqref{enum:EJA_isom_rank3_simple} of \cref{thm:EJA_isom}, $E$ is Jordan-isomorphic to $\calH^3(\bbA)$ for some $\bbA \in \{\bbR,\bbC,\bbH,\bbO\}$, and we let $\phi\colon E\to \calH^3(\bbA)$ be a Jordan isomorphism.
Since $\calS^3$ is a subalgebra of $\calH^3(\bbA)$, we see that $\phi^{-1}(\calS^3)$ is a subalgebra of $E$ that is Jordan-isomorphic to $\calS^3$.
Therefore, by \cref{cor:CP_isom_Jordan,lem:subalg_expface}, $\CP(E_+)$ has the exposed face $\CP(\phi^{-1}(\calS^3)_+)$ that is linearly isomorphic to $\CP(\calS_+^3)$.
\end{proof}

\begin{lemma}\label{lem:CP_PSD3}
$\CP(\calS_+^3)$ is not facially exposed.
\end{lemma}

We defer the proof of \cref{lem:CP_PSD3} to \cref{subsec:proof_lem_rank3_3} so as not to disturb the flow of the argument.
Combining \cref{lem:rank3_3_isom,lem:CP_PSD3}, we obtain the following proposition.

\begin{proposition}\label{prop:rank3_3}
Under the assumption of Case~\ref{enum:rank3_3}, $\CP(E_+)$ is not facially exposed.
\end{proposition}

\begin{proof}
By \cref{lem:rank3_3_isom}, $\CP(E_+)$ has an exposed face that is linearly isomorphic to $\CP(\calS_+^3)$.
Since $\CP(\calS_+^3)$ is not facially exposed by \cref{lem:CP_PSD3}, it follows from \cref{thm:nonexposed,thm:isom_face_exposed} that $\CP(E_+)$ is not facially exposed.
\end{proof}

Summarizing \cref{prop:rank3_111,prop:rank3_12,prop:rank3_3}, we obtain the following characterization of the facial exposedness of $\CP(E_+)$ when $E$ has rank $3$.

\begin{theorem}\label{thm:rank3}
Let $E$ be an EJA of rank $3$.
Then $\CP(E_+)$ is facially exposed if and only if $E$ is not simple.
\end{theorem}

\subsection{Proof of \cref{lem:CP_PSD3}}\label{subsec:proof_lem_rank3_3}
This subsection is devoted to the proof of \cref{lem:CP_PSD3}.
For clarity, we first outline the structure of the proof.
In \cref{subsubsec:CP_PSD3_step1}, we construct an exposed face of $\CP(\calS_+^3)$ and then construct an exposed face of the former face.
By \eqref{enum:face_iff} of \cref{thm:nonexposed}, the second face is also a face of $\CP(\calS_+^3)$.
In \cref{subsubsec:CP_PSD3_step2}, we show that the second face is not exposed in $\CP(\calS_+^3)$, which implies that $\CP(\calS_+^3)$ is not facially exposed.

We use adjugates of matrices in the subsequent discussion, so the following lemma summarizes formulas for adjugates for later use.

\begin{lemma}\label{lem:adjX}
\leavevmode
\begin{enumerate}[label=(\roman*), ref=\roman*]
\item Let $\bm{X} \in \calS^3$.
If $\bm{X}$ has rank at most $1$, then $\adj(\bm{X}) = \bm{O}$. \label{enum:adj_zero}
\item $\{(\adj(\bm{X}))_{ij} \in H_2(\calS^3) \mid 1\le i\le j\le 3\}$ is linearly independent. \label{enum:adj_linearly_indep}
\item Let $\bm{B}\in \calS^3$.
Then the gradient $\nabla \langle\bm{B},\adj(\bm{X})\rangle \in \calS^3$ satisfies
\begin{equation*}
\begin{cases}
(\nabla \langle\bm{B},\adj(\bm{X})\rangle)_{11} = B_{33}X_{22} - 2B_{23}X_{23} + B_{22}X_{33},\\
(\nabla \langle\bm{B},\adj(\bm{X})\rangle)_{12} = -B_{33}X_{12} + B_{23}X_{13} + B_{13}X_{23} -B_{12}X_{33},\\
(\nabla \langle\bm{B},\adj(\bm{X})\rangle)_{13} = B_{23}X_{12} -B_{13}X_{22} -B_{22}X_{13} + B_{12}X_{23}.
\end{cases}
\end{equation*}\label{enum:B_adjX_nabla}
\end{enumerate}
\end{lemma}

\begin{proof}
For $\bm{X}\in \calS^3$, its adjugate is calculated as
\begin{align}
\adj(\bm{X}) &= \begin{pmatrix}
X_{22}X_{33}-X_{23}^2 & X_{13}X_{23} - X_{12}X_{33} & X_{12}X_{23}-X_{22}X_{13}\\
X_{13}X_{23} - X_{12}X_{33} & X_{11}X_{33}-X_{13}^2 & X_{12}X_{13} - X_{11}X_{23} \\
X_{12}X_{23}-X_{22}X_{13} & X_{12}X_{13} - X_{11}X_{23} & X_{11}X_{22}-X_{12}^2
\end{pmatrix} \label{eq:adjX_entries}\\
&= \bm{X}^2 - \tr(\bm{X})\bm{X} + \frac{1}{2}(\tr(\bm{X})^2 - \tr(\bm{X}^2))\bm{I}_3, \label{eq:adjX_tr}
\end{align}
where the second equality follows from \cite[Theorem~1]{Tau1968}.
\eqref{enum:adj_zero} holds since every $2\times 2$ minor of $\bm{X}\in \calS^3$ of rank at most $1$ is $0$.
\eqref{enum:adj_linearly_indep} holds since every $X_{ij}X_{kl}$ for $1\le i\le j\le 3$ and $1\le k\le l\le 3$ appears in the upper-triangular entries of $\adj(\bm{X})$ at most once as shown in \eqref{eq:adjX_entries}.
\eqref{enum:B_adjX_nabla} can be derived from \eqref{eq:adjX_tr}.
\end{proof}

\subsubsection{Construction of a candidate non-exposed face of $\CP(\calS_+^3)$}\label{subsubsec:CP_PSD3_step1}
Let
\begin{equation*}
\bm{G} \coloneqq \begin{pmatrix}
1 & 0 & 1\\
0 & 1 & 0\\
1 & 0 & 0
\end{pmatrix}
\end{equation*}
and let $D_1$ be the element in $\calS(\calS^3)$ such that the associated quadratic form satisfies
\begin{align}
q_{D_1}(\bm{X}) &= \langle \bm{G},\bm{X}\rangle^2 + 4\langle \bm{G}^{-1},\adj(\bm{X})\rangle + 4\bm{e}_1^\top\adj(\bm{X})\bm{e}_1 \label{eq:qD1}\\
&=  (X_{11} + X_{22} + 2X_{13})^2 \nonumber\\
&\quad + 4(X_{11}X_{33}-X_{13}^2 -X_{11}X_{22}+X_{12}^2 + 2X_{12}X_{23} -2X_{22}X_{13}) \nonumber\\
&\quad + 4(X_{22}X_{33}-X_{23}^2).\nonumber
\end{align}

\begin{lemma}\label{lem:rank3_3_D1_COP}
For any $\bm{X}\in \setint\calS_+^3$, we have
\begin{align}
\langle\bm{G},\bm{X}\rangle^2 + 4\langle \bm{G}^{-1},\adj(\bm{X})\rangle &> 0, \label{eq:qD1_1st2nd}\\
\bm{e}_1^\top\adj(\bm{X})\bm{e}_1 &> 0. \label{eq:eq:qd1_3rd}
\end{align}
In particular, it follows that $D_1 \in \COP(\calS_+^3)$.
\end{lemma}

\begin{proof}
Let $\bm{X} \in \setint\calS_+^3$.
Since $\bm{X}$ is positive definite, $\adj(\bm{X})$ can be written as $\adj(\bm{X}) = \det(\bm{X})\bm{X}^{-1}$~\cite[equation~(0.8.2.2)]{HJ2013}.
This implies that $\adj(\bm{X})$ is also positive definite, and thus we have \eqref{eq:eq:qd1_3rd}.

In what follows, we prove \eqref{eq:qD1_1st2nd}.
The eigenvalues of $\bm{G}$ are $\frac{1 \pm \sqrt{5}}{2}$ and $1$, and the inertia of $\bm{G}$ is equal to that of $\sqrt{\bm{X}}\bm{G}\sqrt{\bm{X}}$ by Sylvester's law of inertia~\cite[Theorem~4.5.8]{HJ2013}.
Let $l_1$, $l_2$, $-l_3$ denote the eigenvalues of $\sqrt{\bm{X}}\bm{G}\sqrt{\bm{X}}$, where $l_1,l_2,l_3 > 0$.
Then each of the two terms that appear on the left-hand side of \eqref{eq:qD1_1st2nd} is calculated as follows.
First, we have
\begin{equation*}
\langle \bm{G},\bm{X}\rangle^2 = \tr(\sqrt{\bm{X}}\bm{G}\sqrt{\bm{X}})^2 = (l_1 + l_2 - l_3)^2.
\end{equation*}
Second, we have
\begin{align*}
4\langle \bm{G}^{-1},\adj(\bm{X})\rangle &= 4\det(\bm{X})\tr((\sqrt{\bm{X}}\bm{G}\sqrt{\bm{X}})^{-1})\\
&= -4\det(\sqrt{\bm{X}}\bm{G}\sqrt{\bm{X}})\tr((\sqrt{\bm{X}}\bm{G}\sqrt{\bm{X}})^{-1})\\
&= 4(l_2l_3 + l_1l_3 - l_1l_2),
\end{align*}
where the first equality follows from $\adj(\bm{X}) = \det(\bm{X})\bm{X}^{-1}$ and the second equality follows from $\det(\bm{G}) = -1$.
Therefore, it follows that
\begin{align*}
\langle\bm{G},\bm{X}\rangle^2 + 4\langle \bm{G}^{-1},\adj(\bm{X})\rangle &= (l_1 + l_2 - l_3)^2 + 4(l_2l_3 + l_1l_3 - l_1l_2)\\
&= (l_1 - l_2)^2 + l_3^2 + 2(l_1 + l_2)l_3\\
&> 0,
\end{align*}
so we have \eqref{eq:qD1_1st2nd}.

By \eqref{eq:qD1_1st2nd} and \eqref{eq:eq:qd1_3rd}, we see that $q_{D_1}(\bm{X}) > 0$ for all $\bm{X} \in \setint\calS_+^3$.
The nonnegativity of $q_{D_1}$ on $\calS_+^3$ follows from its continuity.
Therefore, we obtain $D_1 \in \COP(\calS_+^3)$.
\end{proof}

Recall that $Z_{\calS_+^3}(D_1)$, denoted by $Z_1$ for simplicity, is defined in \eqref{eq:def_ZKD}.
By \cref{lem:CP_exposed_face}, $\CP(Z_1)$ is an exposed face of $\CP(\calS_+^3)$.
The following lemma gives a more explicit description of $Z_1$.

\begin{lemma}\label{lem:rank3_3_Z1}
For $t\in \bbR$, we define
\begin{equation}
\bm{X}(t) \coloneqq \begin{pmatrix}
1 & t & -t^2\\
t & 1 & -t\\
-t^2 & -t & t^2
\end{pmatrix} \in \calS^3. \label{eq:def_Xt}
\end{equation}
Then we have
\begin{equation}
Z_1 = \{\bm{x}\bm{x}^\top \mid q_{\bm{G}}(\bm{x}) = 0\} \cup \bbR_+\{\bm{X}(t) \mid t\in (-1,1)\}. \label{eq:rank3_3_Z1}
\end{equation}
\end{lemma}

\begin{proof}
We only prove that $Z_1$ is contained in the right-hand side of \eqref{eq:rank3_3_Z1}.
The converse inclusion follows from the positive semidefiniteness of $\bm{X}(t)$ for every $t\in (-1,1)$ and direct calculations.
Let $\bm{X} \in Z_1$, i.e., $\bm{X} \in \calS_+^3$ and $q_{D_1}(\bm{X}) = 0$.
The rank of $\bm{X}$ must be at most $2$; otherwise, $\bm{X}$ is positive definite, and \cref{lem:rank3_3_D1_COP} implies that $q_{D_1}(\bm{X}) > 0$, which contradicts $\bm{X}\in Z_1$.

First, suppose that $\bm{X}$ has rank at most $1$.
Then there exists $\bm{x}\in \bbR^3$ such that $\bm{X} = \bm{x}\bm{x}^\top$.
It follows from \eqref{enum:adj_zero} of \cref{lem:adjX} that $0 = q_{D_1}(\bm{X}) = q_{\bm{G}}(\bm{x})^2$, i.e., $q_{\bm{G}}(\bm{x}) = 0$.
This means that $\bm{X}$ belongs to the right-hand side of \eqref{eq:rank3_3_Z1}.

Second, suppose that $\bm{X}$ has rank $2$.
Then the rank--nullity theorem implies that the dimension of $\Ker(\bm{X})$ is $1$, and we have $\Ker(\bm{X}) = \bbR\bm{d}$ for some nonzero $\bm{d}\in \bbR^3$.
It follows from the discussion in \cite[section~0.8.2]{HJ2013} that $\adj(\bm{X}) = a\bm{d}\bm{d}^\top$ for some $a > 0$.

The vector $\bm{d}$ satisfies the following.
First, it follows that $d_1 = 0$ since, by $q_{D_1}(\bm{X}) = 0$ and the discussion in \cref{lem:rank3_3_D1_COP}, we have
$0 = \bm{e}_1^\top(\adj(\bm{X}))\bm{e}_1 = ad_1^2$.
Second, it follows that $d_3 \neq 0$.
Indeed, to derive a contradiction, we assume that $d_3 = 0$.
Then $\bm{d} = d_2\bm{e}_2$.
Note that $d_2 \neq 0$ by $\bm{d} \neq \bm{0}$.
Since $\bm{X}$ is a positive semidefinite matrix of rank $2$ and $\Ker(\bm{X}) = \bbR\bm{d} = \bbR\bm{e}_2$, the second row and column of $\bm{X}$ are zero and
$\begin{psmallmatrix}
X_{11} & X_{13}\\
X_{13} & X_{33}
\end{psmallmatrix}$ is positive definite; in particular, we have $X_{11}X_{33} - X_{13}^2 > 0$. 
Therefore, \eqref{eq:qD1} leads to
\begin{equation*}
q_{D_1}(\bm{X}) = (X_{11} + 2X_{13})^2 + 4(X_{11}X_{33} - X_{13}^2) > 0,
\end{equation*}
which contradicts $\bm{X} \in Z_1$.
Thus, we conclude that $\bm{d}$ is a nonzero multiple of $(0,t,1)^\top$ for some $t\in \bbR$.
By
\begin{equation}
\Ker(\bm{X}) = \bbR\bm{d} = \bbR(0,t,1)^\top, \label{eq:rank3_3_KerX}
\end{equation}
 we see that
\begin{equation}
\bm{X} = \begin{pmatrix}
X_{11} & X_{12} & -tX_{12}\\
X_{12} & X_{22} & -tX_{22}\\
-tX_{12} & -tX_{22} & t^2X_{22}
\end{pmatrix}. \label{eq:X_rank2}
\end{equation}
Substituting this into \eqref{eq:qD1}, we have
\begin{align}
 &q_{D_1}(\bm{X}) \nonumber\\
={}& (X_{11} + X_{22} - 2tX_{12})^2 - 4(X_{11}X_{22} - X_{12}^2)(1-t^2) \label{eq:qD1_X_rank2}\\
={}& \tr\left(\begin{pmatrix}
X_{11} & X_{12}\\
X_{12} & X_{22}
\end{pmatrix}\begin{pmatrix}
1 & -t\\
-t & 1
\end{pmatrix}\right)^2  - 4\det\begin{pmatrix}
X_{11} & X_{12}\\
X_{12} & X_{22}
\end{pmatrix}\det\begin{pmatrix}
1 & -t\\
-t & 1
\end{pmatrix}. \label{eq:qD1_X_rank2_mat}
\end{align}
We note that since the positive semidefinite matrix $\bm{X}$ shown in \eqref{eq:X_rank2} has rank $2$ and is congruent to $\{0\}\oplus \begin{psmallmatrix}
X_{11} & X_{12}\\
X_{12} & X_{22}
\end{psmallmatrix}$, the matrix $\begin{psmallmatrix}
X_{11} & X_{12}\\
X_{12} & X_{22}
\end{psmallmatrix}$ is positive definite.
In particular,\begin{equation}
X_{11}X_{22} - X_{12}^2 > 0. \label{eq:X11X22_X12}
\end{equation}

If $\abs{t} > 1$, then \eqref{eq:qD1_X_rank2} and \eqref{eq:X11X22_X12} imply that $q_{D_1}(\bm{X}) > 0$, which contradicts $\bm{X} \in Z_1$.
If $\abs{t} = 1$, since
\begin{equation*}
(X_{11} + X_{22})^2 - \abs{2tX_{12}}^2 \ge 4(X_{11}X_{22} - X_{12}^2) > 0
\end{equation*}
holds by \eqref{eq:X11X22_X12}, we see from \eqref{eq:qD1_X_rank2} that
\begin{equation*}
q_{D_1}(\bm{X}) = (X_{11} + X_{22} - 2tX_{12})^2 > 0,
\end{equation*}
which is also a contradiction.
Now, we assume that $\abs{t} < 1$.
Then $\begin{psmallmatrix}
X_{11} & X_{12}\\
X_{12} & X_{22}
\end{psmallmatrix}$ and $\begin{psmallmatrix}
1 & -t\\
-t & 1
\end{psmallmatrix}$ are positive definite.
In general, for $\bm{A},\bm{B}\in \setint\calS_+^2$, the AM--GM-type inequality $\tr(\bm{A}\bm{B}) \ge 2\sqrt{\det(\bm{A})\det(\bm{B})}$ holds with equality if and only if $\bm{A}$ is a positive multiple of $\bm{B}^{-1}$, which can be proved by applying the usual AM--GM inequality to the eigenvalues of $\sqrt{\bm{B}}\bm{A}\sqrt{\bm{B}}$.
By setting $\bm{A}$ and $\bm{B}$ to $\begin{psmallmatrix}
X_{11} & X_{12}\\
X_{12} & X_{22}
\end{psmallmatrix}$ and $\begin{psmallmatrix}
1 & -t\\
-t & 1
\end{psmallmatrix}$ in that formula, respectively, we see from \eqref{eq:qD1_X_rank2_mat} that $q_{D_1}(\bm{X}) \ge 0$ with equality if and only if $\begin{psmallmatrix}
X_{11} & X_{12}\\
X_{12} & X_{22}
\end{psmallmatrix}$ is a positive multiple of $\begin{psmallmatrix}
1 & -t\\
-t & 1
\end{psmallmatrix}^{-1} = \frac{1}{1-t^2}\begin{psmallmatrix}
1 & t\\
t & 1
\end{psmallmatrix}$.
By $q_{D_1}(\bm{X}) = 0$, there exists $b > 0$ such that $\begin{psmallmatrix}
X_{11} & X_{12}\\
X_{12} & X_{22}
\end{psmallmatrix} = b\begin{psmallmatrix}
1 & t\\
t & 1
\end{psmallmatrix}$.
Substituting this into \eqref{eq:X_rank2}, we obtain $\bm{X} = b\bm{X}(t) \in  \bbR_+\bm{X}(t)$.
\end{proof}

For $t\in \bbR$, we define
\begin{equation*}
\bm{y}(t) \coloneqq (2,2t,-1-t^2)^\top.
\end{equation*}
Then the vectors generating the first component of the right-hand side of \eqref{eq:rank3_3_Z1} can be described as
\begin{equation}
\{\bm{x}\in \bbR^3 \mid q_{\bm{G}}(\bm{x}) = 0\} = \bbR\{\bm{y}(t) \mid t\in \bbR\} \cup \{(0,0,x_3)^\top \mid x_3\in \bbR\}. \label{eq:zero_qG}
\end{equation}
We only prove that the left-hand side of \eqref{eq:zero_qG} is contained in the right-hand side because we can check the converse inclusion by direct calculations.
Let $\bm{x}\in \bbR^3$ be such that
\begin{equation}
q_{\bm{G}}(\bm{x}) = x_1^2 + x_2^2 + 2x_1x_3 = 0. \label{eq:qGx_0}
\end{equation}
First, we assume that $x_1 = 0$.
Since $x_2 = 0$ by \eqref{eq:qGx_0}, we have $\bm{x} = (0,0,x_3)^\top$, so it belongs to the right-hand side of \eqref{eq:zero_qG}.
Second, we assume that $x_1 \neq 0$.
Letting $t \coloneqq x_2 / x_1$, we have $\bm{x} = \tfrac{x_1}{2}\bm{y}(t)$, which also belongs to the right-hand side of \eqref{eq:zero_qG}.

Next, we construct an exposed face of $\CP(Z_1)$.
Let
\begin{equation*}
\bm{H} \coloneqq \begin{pmatrix}
5 & 10 & 13 \\
10 & -5 & 16\\
13 & 16 & 32
\end{pmatrix}
\end{equation*}
and let $D_2$ be the element in $\calS(\calS^3)$ such that the associated quadratic form satisfies
\begin{equation*}
q_{D_2}(\bm{X}) = \langle\bm{H},\bm{X}\rangle^2 - 108\adj(\bm{X})_{22}.
\end{equation*}

\begin{lemma}
We have $D_2 \in \COP(Z_1)$.
Furthermore, for $t\in \bbR$, we define $\bm{Y}(t) \coloneqq \bm{y}(t)\bm{y}(t)^\top$.
Then $Z_{Z_1}(D_2)$, denoted by $Z_2$ for simplicity, is given by
\begin{equation}
Z_2 = \bbR_+\{\bm{Y}(-\tfrac{1}{2}),\bm{Y}(0),\bm{Y}(\tfrac{1}{2}),\bm{Y}(2),\bm{X}(0),\bm{X}(\tfrac{1}{2})\}, \label{eq:rank3_3_Z2}
\end{equation}
where $\bm{X}(\cdot)$ is defined in \eqref{eq:def_Xt}.
\end{lemma}

\begin{proof}
In the first step, we show that $D_2 \in \COP(Z_1)$.
Let $\bm{X} \in Z_1$.
It follows from \cref{lem:rank3_3_Z1} that $\bm{X}$ belongs to the first or second component on the right-hand side of \eqref{eq:rank3_3_Z1}.
First, suppose that $\bm{X} = \bm{x}\bm{x}^\top$ for some $\bm{x}\in \bbR^3$ with $q_{\bm{G}}(\bm{x}) = 0$.
By \eqref{enum:adj_zero} of \cref{lem:adjX}, we see that
\begin{equation}
q_{D_2}(\bm{X}) = q_{\bm{H}}(\bm{x})^2 \ge 0. \label{eq:rank3_3_X_rank1}
\end{equation}
Second, suppose that $\bm{X} = a\bm{X}(t)$ for some $a\ge 0$ and $t\in (-1,1)$.
Then we have
\begin{equation}
q_{D_2}(\bm{X}) = 36a^2t^2(2t-1)^2 \ge 0. \label{eq:rank3_3_X_rank2}
\end{equation}
Therefore, we obtain $D_2 \in \COP(Z_1)$.

In the second step, to prove \eqref{eq:rank3_3_Z2}, we only show that $Z_2$ is contained in the right-hand side of \eqref{eq:rank3_3_Z2} because the converse inclusion follows from \cref{lem:rank3_3_Z1}, \eqref{eq:zero_qG}, and direct calculations.
First, if the inequality in \eqref{eq:rank3_3_X_rank1} holds with equality, we have $q_{\bm{H}}(\bm{x}) = 0$.
By $q_{\bm{G}}(\bm{x}) = 0$ and \eqref{eq:zero_qG}, $\bm{x}$ belongs to the first or second component on the right-hand side of \eqref{eq:zero_qG}.
If $\bm{x} = (0,0,x_3)^\top$, then since $q_{\bm{H}}(\bm{x}) = 32x_3^2$, the equality $q_{\bm{H}}(\bm{x}) = 0$ implies $\bm{x} = \bm{0}$.
When $\bm{x} = \bm{0}$,  we have $\bm{X} = \bm{O}$, which belongs to the right-hand side of \eqref{eq:rank3_3_Z2}.
If $\bm{x} = a\bm{y}(t)$ for some $a \in \bbR$ and $t\in \bbR$, then since
\begin{equation}
q_{\bm{H}}(\bm{y}(t)) = 32\left(t + \frac{1}{2}\right)t\left(t - \frac{1}{2}\right)\left(t - 2\right), \label{eq:xC1x}
\end{equation}
the equality $q_{\bm{H}}(\bm{x}) = 0$ implies $a = 0$ or $t\in \{-\tfrac{1}{2},0,\tfrac{1}{2},2\}$; in either case, $\bm{X}$ belongs to the right-hand side of \eqref{eq:rank3_3_Z2}.
Second, if the inequality in \eqref{eq:rank3_3_X_rank2} holds with equality, we have $a = 0$ or $t \in \{0,\tfrac{1}{2}\}$; in either case, $\bm{X}$ belongs to the right-hand side of \eqref{eq:rank3_3_Z2}.
Thus, we obtain \eqref{eq:rank3_3_Z2}.
\end{proof}

We see from \cref{lem:CP_exposed_face} that
\begin{equation}
\CP(Z_1) \cap \{D_2\}^\perp = \CP(Z_2). \label{eq:CPZ2_exposed_by_D2}
\end{equation}
Since $\CP(Z_1)$ is a face of $\CP(\calS_+^3)$, by \eqref{enum:face_iff} of \cref{thm:nonexposed}, $\CP(Z_2)$ is a face of $\CP(\calS_+^3)$.

\subsubsection{Non-exposedness of $\CP(Z_2)$}\label{subsubsec:CP_PSD3_step2}
We begin by establishing the following lemma, which is crucial for proving that the face $\CP(Z_2)$ of $\CP(\calS_+^3)$ is not exposed.
\begin{lemma}\label{lem:COPPSD3_cap_CPZ2perp}
It follows that
\begin{equation*}
\COP(\calS_+^3) \cap \CP(Z_2)^\perp\\
\subseteq \{\bm{x}\bm{x}^\top\otimes \bm{x}\bm{x}^\top \mid \text{$\bm{x}\in \bbR^3$ satisfies $q_{\bm{G}}(\bm{x}) = 0$}\}^\perp. 
\end{equation*}
\end{lemma}

\begin{proof}
Let $A \in \COP(\calS_+^3) \cap \CP(Z_2)^\perp$.
By $A \in \CP(Z_2)^\perp$ and \eqref{eq:rank3_3_Z2}, we have
\begin{equation}
q_A(\bm{X}) = 0 \text{ for all $\bm{X} \in \{\bm{Y}(-\tfrac{1}{2}),\bm{Y}(0),\bm{Y}(\tfrac{1}{2}),\bm{Y}(2),\bm{X}(0),\bm{X}(\tfrac{1}{2})\}$.}
\label{eq:qA_eq_0}
\end{equation}
Let $p(\bm{x}) \coloneqq q_A(\bm{x}\bm{x}^\top)$.
Since $p$ belongs to $H_4(\bbR^3)$ and $p(\bm{x}) \ge 0$ for all $\bm{x}\in \bbR^3$ by $A\in \COP(\calS_+^3)$, it follows from \cite{Hil1888} that $p$ is a sum of squares; there exist $p_1,\dots,p_k\in H_2(\bbR^3)$ such that
\begin{equation}
p(\bm{x}) = \sum_{i=1}^k p_i(\bm{x})^2. \label{eq:p_sos}
\end{equation}

Every $p_i$ for $i \in \{1,\dots,k\}$ belongs to $\setspan\{q_{\bm{G}},q_{\bm{H}}\}$.
It is sufficient to show this for $i = 1$.
By $q_A(\bm{Y}(t)) = p(\bm{y}(t))$, \eqref{eq:qA_eq_0}, and \eqref{eq:p_sos}, we have $(p_1\circ \bm{y})(t) = 0$ for all $t\in \{-\tfrac{1}{2},0,\tfrac{1}{2},2\}$.
By the factor theorem, there exists $\lambda \in \bbR$ such that
\begin{equation*}
(p_1\circ \bm{y})(t) = \lambda\left(t + \frac{1}{2}\right)t\left(t-\frac{1}{2}\right)(t-2) = \frac{\lambda}{32}(q_{\bm{H}}\circ \bm{y})(t),
\end{equation*}
where we use \eqref{eq:xC1x} to derive the second equality.
Let $\bm{C} \in \calS^3$ be such that $q_{\bm{C}} = p_1 - \tfrac{\lambda}{32} q_{\bm{H}}$.
Then we see that $(q_{\bm{C}}\circ \bm{y})(t)$ is the zero polynomial.
Calculating $(q_{\bm{C}}\circ \bm{y})(t)$ yields the following identity:
\begin{equation*}
C_{33}t^4 -4C_{23}t^3 + 2(2C_{22} + C_{33} - 2C_{13})t^2 + 4(2C_{12} - C_{23})t + 4C_{11} + C_{33} - 4C_{13} = 0.
\end{equation*}
Comparing the coefficients yields $C_{11} = C_{22} = C_{13}$, denoted by $c \in \bbR$, and $C_{12} = C_{23} = C_{33} = 0$.
Therefore, we obtain $q_{\bm{C}}(\bm{x}) = c(x_1^2 + x_2^2 + 2x_1x_3) = c q_{\bm{G}}(\bm{x})$, and $p_1 = c q_{\bm{G}} + \frac{\lambda}{32} q_{\bm{H}} \in \setspan\{q_{\bm{G}},q_{\bm{H}}\}$.

For each $i \in \{1,\dots,k\}$, there exist $a_i,b_i\in \bbR$ such that $p_i = a_iq_{\bm{G}} + b_iq_{\bm{H}}$.
We define
\begin{equation}
\begin{pmatrix}
\alpha & \beta\\
\beta & \gamma
\end{pmatrix} \coloneqq \sum_{i=1}^k\begin{pmatrix}
a_i\\
b_i
\end{pmatrix}\begin{pmatrix}
a_i\\
b_i
\end{pmatrix}^\top \in \calS_+^2. \label{eq:def_abg}
\end{equation}
Then \eqref{eq:p_sos} leads to
\begin{equation}
p(\bm{x}) = \alpha q_{\bm{G}}(\bm{x})^2 + 2\beta q_{\bm{G}}(\bm{x})q_{\bm{H}}(\bm{x}) + \gamma q_{\bm{H}}(\bm{x})^2. \label{eq:p_gram}
\end{equation}

Using $\alpha$, $\beta$, and $\gamma$ defined in \eqref{eq:def_abg}, we can represent $q_A$ as
\begin{equation}
q_A(\bm{X}) = \alpha\langle\bm{G},\bm{X}\rangle^2 + 2\beta\langle\bm{G},\bm{X}\rangle\langle\bm{H},\bm{X}\rangle + \gamma\langle\bm{H},\bm{X}\rangle^2 + \langle\bm{B},\adj(\bm{X})\rangle \label{eq:qAX}
\end{equation}
for some $\bm{B} \in \calS^3$.
Indeed, let $f\colon H_2(\calS^3) \to H_4(\bbR^3)$ be the linear mapping sending a polynomial $q(\bm{X})$ to the polynomial $q(\bm{x}\bm{x}^\top)$ in $\bm{x}$.
We see that $f$ is surjective, since every monomial $x_ix_jx_kx_l$ for $1\le i\le j\le k\le l\le 3$ is realized by $f(X_{ij}X_{kl})$.
In addition, the dimensions of $H_2(\calS^3)$ and $H_4(\bbR^3)$ are $21$ and $15$, respectively.
The rank--nullity theorem implies that the dimension of $\Ker(f)$, the kernel of $f$, is $6$.
Since the set $\{(\adj(\bm{X}))_{ij} \mid 1\le i\le j\le 3\}$ consists of six elements in $\Ker(f)$ by \eqref{enum:adj_zero} of \cref{lem:adjX} and is linearly independent by \eqref{enum:adj_linearly_indep} of \cref{lem:adjX}, it is a basis for $\Ker(f)$.
In other words, we obtain the following equation:
\begin{equation}
\Ker(f) = \{\langle\bm{B},\adj(\bm{X})\rangle \mid \bm{B} \in \calS^3\}. \label{eq:Kerf}
\end{equation}
Now, we define
\begin{equation*}
q(\bm{X}) \coloneqq q_A(\bm{X}) - (\alpha\langle\bm{G},\bm{X}\rangle^2 + 2\beta\langle\bm{G},\bm{X}\rangle\langle\bm{H},\bm{X}\rangle + \gamma\langle\bm{H},\bm{X}\rangle^2).
\end{equation*}
It follows from the definition of $p$ and \eqref{eq:p_gram} that
\begin{equation*}
q(\bm{x}\bm{x}^\top) = p(\bm{x}) - (\alpha q_{\bm{G}}(\bm{x})^2 + 2\beta q_{\bm{G}}(\bm{x})q_{\bm{H}}(\bm{x}) + \gamma q_{\bm{H}}(\bm{x})^2) = 0,
\end{equation*}
i.e., $q\in \Ker(f)$.
Combining this with \eqref{eq:Kerf}, we obtain \eqref{eq:qAX}.

Both $\beta$ and $\gamma$ defined in \eqref{eq:def_abg} are in fact $0$.
To see this, fix $t\in \{0,\tfrac{1}{2}\}$ arbitrarily.
Since $q_A(\bm{X})$ is nonnegative on $\calS_+^3$ and it attains $0$ at $\bm{X} = \bm{X}(t)$ by \eqref{eq:qA_eq_0}, the gradient of $q_A$ at $\bm{X}(t)$, i.e., $\nabla q_A(\bm{X}(t))$, satisfies
\begin{align}
-\nabla q_A(\bm{X}(t)) &\overset{\scriptsize \text{(a)}}\in \{\bm{Y} \in \calS^3 \mid \langle\bm{Y},\bm{X}-\bm{X}(t)\rangle \le 0 \text{ for all $\bm{X}\in \calS_+^3$}\} \label{eq:prod_qAXt_Xt}\\
& \overset{\scriptsize \text{(b)}}= -\calS_+^3 \cap \{\bm{X}(t)\}^\perp \nonumber\\
& \overset{\scriptsize \text{(c)}}= -\{\bm{Y}\in \calS_+^3 \mid \bm{Y}\bm{X}(t) = \bm{X}(t)\bm{Y} = \bm{O}\},\nonumber
\end{align}
where we use the first-order optimality condition~\cite[Theorem~27.4]{Roc1970} to derive (a), use the formula for the normal cone of $\calS_+^3$~\cite[Theorem~2.1]{Fle1985} to derive (b), and use $\bm{X}(t) \in \calS_+^3$ to derive (c).
It follows from \eqref{eq:prod_qAXt_Xt} and $\Ker(\bm{X}(t)) = \bbR(0,t,1)^\top$ (see \eqref{eq:rank3_3_KerX}) that $\nabla q_A(\bm{X}(t))$ is a (nonnegative) multiple of $\begin{psmallmatrix}
0 & 0 & 0\\
0 & t^2 & t\\
0 & t & 1
\end{psmallmatrix}$.
In particular, it follows that\begin{equation}
(\nabla q_A(\bm{X}(t)))_{1j} = 0 \text{ for all $j \in \{1,2,3\}$.} \label{eq:nabla_qA_0}
\end{equation}
From \eqref{eq:qAX}, we have
\begin{multline}
\label{eq:nabla_qAX_calc} \nabla q_A(\bm{X})\\
= 2(\alpha\langle\bm{G},\bm{X}\rangle + \beta\langle\bm{H},\bm{X}\rangle)\bm{G} + 2(\beta\langle\bm{G},\bm{X}\rangle + \gamma\langle\bm{H},\bm{X}\rangle)\bm{H} + \nabla\langle\bm{B},\adj(\bm{X})\rangle. 
\end{multline}
We substitute $\bm{X}(0)$ for $\bm{X}$ in \eqref{eq:nabla_qAX_calc}.
It follows from \cref{eq:nabla_qA_0} and \eqref{enum:B_adjX_nabla} of \cref{lem:adjX} that $0 = (\nabla q_A(\bm{X}(0)))_{12} = 40\beta$, which implies that $\beta = 0$.
With this, by substituting $\bm{X}(\tfrac{1}{2})$ for $\bm{X}$ in \eqref{eq:nabla_qAX_calc}, we see again from \eqref{eq:nabla_qA_0} and \eqref{enum:B_adjX_nabla} of \cref{lem:adjX} that
\begin{equation*}
0 = -(\nabla q_A(\bm{X}(\tfrac{1}{2})))_{11} - 2(\nabla q_A(\bm{X}(\tfrac{1}{2})))_{12} + (\nabla q_A(\bm{X}(\tfrac{1}{2})))_{13} = 108\gamma,
\end{equation*}
which implies that $\gamma = 0$.

It follows from \eqref{eq:p_gram} and $\beta = \gamma = 0$ that $p(\bm{x}) = \alpha q_{\bm{G}}(\bm{x})^2$.
Using this equality, for any $\bm{x}\in \bbR^3$ with $q_{\bm{G}}(\bm{x}) = 0$, we have
\begin{equation*}
\langle A,\bm{x}\bm{x}^\top\otimes \bm{x}\bm{x}^\top\rangle = q_A(\bm{x}\bm{x}^\top) = p(\bm{x}) = 0.
\end{equation*}
Therefore, we obtain the desired result.
\end{proof}

We are ready to show the non-exposedness of $\CP(Z_2)$.
To derive a contradiction, we assume that $\CP(Z_2)$ is an exposed face of $\CP(\calS_+^3)$.
Then there exists $D\in \COP(\calS_+^3)$ such that $\CP(Z_2) = \CP(\calS_+^3) \cap \{D\}^\perp$.
Since $D\in \COP(\calS_+^3) \cap \CP(Z_2)^\perp$, it follows from \cref{lem:COPPSD3_cap_CPZ2perp} that
\begin{equation}
\{\bm{x}\bm{x}^\top\otimes \bm{x}\bm{x}^\top \mid \text{$\bm{x}\in \bbR^3$ satisfies $q_{\bm{G}}(\bm{x}) = 0$}\} \subseteq \CP(\calS_+^3) \cap \{D\}^\perp = \CP(Z_2).  \label{eq:qG_zero_outer_prod}
\end{equation}
Let $\bm{x}^* \coloneqq (2,2,-2)^\top$ and $\bm{X}^* \coloneqq \bm{x}^*(\bm{x}^*)^\top$.
It follows from $q_{\bm{G}}(\bm{x}^*) = 0$ and \eqref{eq:qG_zero_outer_prod} that $\bm{X}^*\otimes \bm{X}^* \in \CP(Z_2)$.
However, $\bm{X}^*$ satisfies
\begin{equation*}
\langle D_2,\bm{X}^*\otimes \bm{X}^*\rangle = q_{D_2}(\bm{X}^*) = ((\bm{x}^*)^\top\bm{H}\bm{x}^*)^2 = 576 > 0,
\end{equation*}
where we use \eqref{enum:adj_zero} of \cref{lem:adjX} to derive the second equality.
By \eqref{eq:CPZ2_exposed_by_D2}, we have $\bm{X}^*\otimes \bm{X}^* \not\in \CP(Z_2)$, which is a contradiction.
Thus, $\CP(Z_2)$ is a non-exposed face of $\CP(\calS_+^3)$, and this completes the proof of \cref{lem:CP_PSD3}.

\section{Case of rank $4$}\label{sec:rank4}
Let $E$ be an EJA of rank $4$.
We classify the facial exposedness of $\CP(E_+)$ depending on whether $E$ is Jordan-isomorphic to the Hadamard EJA $\bbR^4$.

First, we discuss the case where $E$ is not Jordan-isomorphic to $\bbR^4$.
 
\begin{lemma}\label{lem:rank4_not_HEJA_expface}
Let $(E,\circ)$ be an EJA of rank $4$ that is not Jordan-isomorphic to the Hadamard EJA $\bbR^4$.
Then $\CP(E_+)$ has an exposed face that is linearly isomorphic to $\CP(\bbR_+^2\times L_+^n)$ for some $n\ge 3$.
\end{lemma}

\begin{proof}
Fix a Jordan frame $\{c_1,c_2,c_3,c_4\}$ of $E$.
By \cite[Theorem~\RNum{4}.2.1]{FK1994}, $E$ can be decomposed into the following direct sum of subspaces:
\begin{equation*}
E = \bbR c_1 \oplus \bbR c_2 \oplus \bbR c_3 \oplus \bbR c_4 \oplus E_{12} \oplus E_{13} \oplus E_{23} \oplus E_{14} \oplus E_{24} \oplus E_{34},
\end{equation*}
where $E_{ij} \coloneqq \{x \in E \mid c_i\circ x = c_j\circ x = \tfrac{1}{2}x\}$ for every $1\le i < j\le 4$.
By the assumption that $E$ is not Jordan-isomorphic to the Hadamard EJA $\bbR^4$, there exist $1\le i <  j\le 4$ such that $E_{ij} \neq \{0\}$.
We may assume that $E_{34} \neq \{0\}$ without loss of generality.
We define $E_{3:4} \coloneqq \bbR c_3 \oplus \bbR c_4 \oplus E_{34}$, which is a  subalgebra of $E$ by \cite[Lemma~20.a]{GST2004} and \cite[Proposition~\RNum{4}.1.1]{FK1994}.
Since $E_{3:4}$ is a simple EJA of rank $2$ by $E_{34} \neq \{0\}$ and \cite[page~168]{NN2016}, \eqref{enum:EJA_isom_rank2} of \cref{thm:EJA_isom} implies that $E_{3:4}$ is Jordan-isomorphic to a Jordan spin algebra $L^n$ for some $n\ge 3$.
Therefore, $E' \coloneqq \bbR c_1 \oplus \bbR c_2 \oplus E_{3:4}$ is a subalgebra of $E$ and $E'$ is Jordan-isomorphic to the direct product of the Hadamard EJA $\bbR^2$ and $L^n$.
Since the cone of squares in $\bbR^2\times L^n$ is $\bbR_+^2 \times L_+^n$, it follows from \cref{cor:CP_isom_Jordan,lem:subalg_expface} that $\CP(E_+)$ has the exposed face $\CP(E_+')$ that is linearly isomorphic to  $\CP(\bbR_+^2 \times L_+^n)$.
\end{proof}

\begin{lemma}\label{lem:nno2_soc3}
$\CP(\bbR_+^2 \times L_+^n)$ is not facially exposed for any $n\ge 3$.
\end{lemma}

We defer the proof of \cref{lem:nno2_soc3} to \cref{subsec:proof_lem_nno2_soc3} so as not to disturb the flow of the argument.
\Cref{lem:rank4_not_HEJA_expface,lem:nno2_soc3} lead to the non-facial exposedness of $\CP(E_+)$ in the case where $E$ is not Jordan-isomorphic to $\bbR^4$, as stated in the following proposition.

\begin{proposition}\label{prop:rank4_not_HEJA}
Let $E$ be an EJA of rank $4$ that is not Jordan-isomorphic to the Hadamard EJA $\bbR^4$.
Then $\CP(E_+)$ is not facially exposed.
\end{proposition}

\begin{proof}
By \cref{lem:rank4_not_HEJA_expface}, $\CP(E_+)$ has an exposed face that is linearly isomorphic to $\CP(\bbR_+^2 \times L_+^n)$ for some $n\ge 3$.
Since $\CP(\bbR_+^2 \times L_+^n)$ is not facially exposed by \cref{lem:nno2_soc3}, it follows from \cref{thm:nonexposed,thm:isom_face_exposed} that $\CP(E_+)$ is not facially exposed.
\end{proof}

Combining \cref{prop:rank4_not_HEJA} with the case where $E$ is Jordan-isomorphic to $\bbR^4$, we obtain the following characterization of the facial exposedness of $\CP(E_+)$ when $E$ has rank $4$.

\begin{theorem}\label{thm:rank4}
Let $E$ be an EJA of rank $4$.
Then $\CP(E_+)$ is facially exposed if and only if $E$ is Jordan-isomorphic to the Hadamard EJA $\bbR^4$.
\end{theorem}

\begin{proof}
If $E$ is not Jordan-isomorphic to the Hadamard EJA $\bbR^4$, the non-facial exposedness of $\CP(E_+)$ follows from \cref{prop:rank4_not_HEJA}.
If $E$ is Jordan-isomorphic to $\bbR^4$, the facial exposedness of $\CP(E_+)$ can be shown by an argument similar to the proof of \cref{thm:rank1}.
\end{proof}

\subsection{Proof of \cref{lem:nno2_soc3}}\label{subsec:proof_lem_nno2_soc3}
This subsection is devoted to the proof of \cref{lem:nno2_soc3}.
The outline of the proof is similar to that of \cref{lem:CP_PSD3} presented in \cref{subsec:proof_lem_rank3_3}.
In \cref{subsubsec:nno2_soc3_step1}, we observe that $\CP(\bbR_+^2 \times L_+^n)$ has a face that is linearly isomorphic to $\CP(\bbR_+^2 \times L_+^3)$ and construct a face of $\CP(\bbR_+^2 \times L_+^3)$ by constructing a chain of faces.
In \cref{subsubsec:nno2_soc3_step2}, we show that the final face is not exposed in $\CP(\bbR_+^2 \times L_+^3)$, which implies that $\CP(\bbR_+^2 \times L_+^n)$ is not facially exposed.

\subsubsection{Construction of a candidate non-exposed face of $\CP(\bbR_+^2 \times L_+^3)$}\label{subsubsec:nno2_soc3_step1}
In the first step, we show that $\CP(\bbR_+^2 \times L_+^n)$ has an exposed face that is linearly isomorphic to $\CP(\bbR_+^2 \times L_+^3)$.
We define $X \coloneqq \setspan\{\bm{e}_1,\dots,\bm{e}_5\} \subseteq \bbR^{n+2}$.
It follows from \cref{thm:CP_face_subsp} that $\CP((\bbR_+^2 \times L_+^n) \cap X)$ is an exposed face of $\CP(\bbR_+^2 \times L_+^n)$.
Since
\begin{equation*}
(\bbR_+^2 \times L_+^n) \cap X = \{(x_1,\dots,x_5,0,\dots,0)^\top \in \bbR^{n+2} \mid (x_1,\dots,x_5)^\top \in \bbR_+^2 \times L_+^3\}
\end{equation*}
and it is linearly isomorphic to $\bbR_+^2 \times L_+^3$, it follows from \cref{lem:CP_isom} that $\CP((\bbR_+^2 \times L_+^n) \cap X)$ is linearly isomorphic to $\CP(\bbR_+^2 \times L_+^3)$.

In the second step, we construct an exposed face of $\CP(\bbR_+^2 \times L_+^3)$.
Let $\bm{D}_1 \in \calS^5$ be such that the associated quadratic form satisfies
\begin{equation}
q_{\bm{D}_1}(\bm{x}) = \bm{x}_{3:5}^\top\bm{J}_3\bm{x}_{3:5} + (x_1 - x_2 + x_3 + x_4 - x_5)^2 + 2x_1(x_3-x_5). \label{eq:nno2_soc3_qD1}
\end{equation}
We see that $\bm{D}_1 \in \COP(\bbR_+^2 \times L_+^3)$ since each of the three terms on the right-hand side of \eqref{eq:nno2_soc3_qD1} is nonnegative for any $\bm{x}\in \bbR_+^2 \times L_+^3$.
Recall that $Z_{\bbR_+^2\times L_+^3}(\bm{D}_1)$, denoted by $Z_1$ for simplicity, is defined in \eqref{eq:def_ZKD}.
It follows from \cref{lem:CP_exposed_face} that $\CP(Z_1)$ is an exposed face of $\CP(\bbR_+^2 \times L_+^3)$.
The following lemma gives a more explicit description of $Z_1$.

\begin{lemma}\label{lem:nno2_soc3_Z1}
Let
\begin{align*}
\bm{x}(s) &\coloneqq \left(0,1-s,\frac{1+s^2}{2},\frac{1-s^2}{2},s\right)^\top\ (s\in \bbR),\\
\bm{x}^* &\coloneqq (0,0,1,-1,0)^\top
\end{align*}
and define
\begin{equation*}
C \coloneqq \{\alpha(\bm{e}_1 + \bm{e}_2) + \beta\bm{x}(1) \mid \alpha,\beta \ge 0\} = \{(\alpha,\alpha,\beta,0,\beta)^\top \mid \alpha,\beta\ge 0\}.
\end{equation*}
Then we have
\begin{equation}
Z_1 = \bbR_+\{\bm{x}(s) \mid s\le 1\} \cup \bbR_+\bm{x}^* \cup C. \label{eq:nno2_soc3_Z1}
\end{equation}
\end{lemma}

\begin{proof}
We only prove that $Z_1$ is contained in the right-hand side of \eqref{eq:nno2_soc3_Z1} because the converse inclusion can be verified through direct calculations.
Let $\bm{x}\in Z_1$, i.e., $\bm{x}\in \bbR_+^2 \times L_+^3$ and $q_{\bm{D}_1}(\bm{x}) = 0$.
Since $q_{\bm{D}_1}(\bm{x}) = 0$, each of the three terms on the right-hand side of \eqref{eq:nno2_soc3_qD1} is also $0$, and we have
\begin{align}
(x_3 + x_4)(x_3 - x_4) &= x_5^2, \label{eq:qD1x_1st}\\
x_1 - x_2 + x_3 + x_4 - x_5 &= 0, \label{eq:qD1x_2nd}\\
x_1(x_3-x_5) &= 0. \label{eq:qD1x_3rd}
\end{align}

First, we consider the case where $\beta \coloneqq x_3 + x_4$ is not $0$.
We note that $\beta > 0$ by $\bm{x}_{3:5}\in L_+^3$.
Let $s_{\beta} \coloneqq x_5/\beta$.
Then \eqref{eq:qD1x_1st} implies that $x_3 - x_4 = \beta s_{\beta}^2$.
Combining this with $x_3 + x_4 = \beta$, we have
\begin{equation*}
\bm{x}_{3:5} = \beta\left(\frac{1+s_{\beta}^2}{2},\frac{1-s_{\beta}^2}{2},s_{\beta}\right)^\top.
\end{equation*}
We further distinguish two cases depending on whether $s_{\beta} = 1$ or not.
If $s_{\beta} = 1$, then we have $\bm{x}_{3:5} = \beta(1,0,1)^\top$.
Substituting this into \eqref{eq:qD1x_2nd}, we have $x_1 = x_2$, and this common value is nonnegative by $\bm{x}_{1:2} \in \bbR_+^2$.
Then we have
\begin{equation*}
\bm{x} = (x_1,x_1,\beta,0,\beta)^\top \in C.
\end{equation*}
If $s_{\beta}\neq 1$, then since $x_3 - x_5 = \tfrac{\beta}{2}(1-s_{\beta})^2 \neq 0$, it follows from \eqref{eq:qD1x_3rd} that $x_1 = 0$.
Consequently, \eqref{eq:qD1x_2nd} implies that $x_2 = \beta(1-s_{\beta})$.
Since $\bm{x}_{1:2} \in \bbR_+^2$, we have $s_{\beta}\le 1$.
Therefore, we obtain
\begin{equation*}
\bm{x} = \beta\bm{x}(s_{\beta}) \in \bbR_+\{\bm{x}(s) \mid s\le 1\}.
\end{equation*}

Second, we consider the case where $x_3 + x_4 = 0$.
Since \eqref{eq:qD1x_1st} implies that $x_5 = 0$, we see that $\bm{x}_{3:5} = x_3(1,-1,0)$.
Substituting this into \eqref{eq:qD1x_2nd} and \eqref{eq:qD1x_3rd}, we have $x_1 = x_2$ and $x_1x_3 = 0$, respectively.
We note that $x_3 \ge 0$ by $\bm{x}_{3:5} \in L_+^3$.
If $x_3 > 0$, then $x_1 = 0$, and we have
\begin{equation*}
\bm{x} = x_3(0,0,1,-1,0)^\top \in \bbR_+\bm{x}^*.
\end{equation*}
If $x_3 = 0$, then we have
\begin{equation*}
\bm{x} = (x_1,x_1,0,0,0)^\top \in C.
\end{equation*}
Therefore, we obtain the desired result.
\end{proof}

In the third step, we construct an exposed face of $\CP(Z_1)$.
Let $\bm{D}_2 \in \calS^5$ be such that the associated quadratic form satisfies
\begin{equation*}
q_{\bm{D}_2}(\bm{x}) = (x_3 - x_4)(x_3 + x_4 - x_5).
\end{equation*}
Using \cref{lem:nno2_soc3_Z1}, we can easily check that $\bm{D}_2 \in \COP(Z_1)$.
For simplicity, we let $Z_2 \coloneqq Z_{Z_1}(\bm{D}_2)$.
By \cref{lem:CP_exposed_face}, we have
\begin{equation}
\CP(Z_1) \cap \{\bm{D}_2\}^\perp = \CP(Z_2). \label{eq:nno2_soc3_CPZ2_exposed_by_D2}
\end{equation}
Since $\CP(Z_1)$ is a face of $\CP(\bbR_+^2 \times L_+^3)$, by \eqref{enum:face_iff} of \cref{thm:nonexposed}, $\CP(Z_2)$ is a face of $\CP(\bbR_+^2 \times L_+^3)$.

We can represent $Z_2$ more explicitly as
\begin{equation}
Z_2 = \bbR_+\{\bm{x}(0),\bm{x}^*\} \cup C. \label{eq:nno2_soc3_Z2}
\end{equation}
We only show that $Z_2$ is contained in the right-hand side of \eqref{eq:nno2_soc3_Z2} because  the converse inclusion follows from \cref{lem:nno2_soc3_Z1} and direct calculations.
Let $\bm{x} \in Z_2$, i.e., $\bm{x} \in Z_1$ and $q_{\bm{D}_2}(\bm{x}) = 0$.
By \cref{lem:nno2_soc3_Z1}, we see that $\bm{x}$ belongs to one of the three components on the right-hand side of \eqref{eq:nno2_soc3_Z1}.
First, if $\bm{x}\in \bbR_+\bm{x}^*$ or $\bm{x}\in C$, then $\bm{x}$ belongs to the right-hand side of \eqref{eq:nno2_soc3_Z2}.
Next, we assume that there exist $\beta \ge 0$ and $s\le 1$ such that $\bm{x} = \beta\bm{x}(s)$.
Since $0 = q_{\bm{D}_2}(\bm{x}) = \beta^2s^2(1-s)$, we have $\beta = 0$, $s = 0$, or $s = 1$.
If $\beta = 0$, then $\bm{x} = \bm{0}$, which belongs to the right-hand side of  \eqref{eq:nno2_soc3_Z2}.
If $s = 0$, then $\bm{x} = \beta\bm{x}(0) \in \bbR_+\bm{x}(0)$.
If $s = 1$, then $\bm{x} = \beta\bm{x}(1) \in C$.
Therefore, we obtain \eqref{eq:nno2_soc3_Z2}.

\subsubsection{Non-exposedness of $\CP(Z_2)$}\label{subsubsec:nno2_soc3_step2}
To show that $\CP(Z_2)$ is a non-exposed face of $\CP(\bbR_+^2 \times L_+^3)$ by contradiction, we assume that $\CP(Z_2)$ is an exposed face of $\CP(\bbR_+^2 \times L_+^3)$.
Then there exists $\bm{D}\in \COP(\bbR_+^2 \times L_+^3)$ such that
\begin{equation}
\CP(Z_2) = \CP(\bbR_+^2 \times L_+^3) \cap \{\bm{D}\}^\perp. \label{eq:nno2_soc3_CPZ2_exposed}
\end{equation}
We define $p(s) \coloneqq q_{\bm{D}}(\bm{x}(s))$, which is a polynomial in $s$.

\begin{lemma}\label{lem:p_zeropoly}
$p$ is the zero polynomial.
\end{lemma}

\begin{proof}
We observe that $s^2$ divides $p(s)$, i.e., $p(0) = p'(0) = 0$.
Indeed, $p(0) = q_{\bm{D}}(\bm{x}(0)) = 0$ by $\bm{x}(0) \in Z_2$ (see \eqref{eq:nno2_soc3_Z2}) and \eqref{eq:nno2_soc3_CPZ2_exposed}.
We see that $p(s) \ge 0$ for any $s\le 1$ since $\bm{x}(s) \in \bbR_+^2\times L_+^3$ (\cref{lem:nno2_soc3_Z1}) and $\bm{D}\in \COP(\bbR_+^2 \times L_+^3)$.
Therefore, $s = 0$ is a local minimizer of $p(s)$, and we have $p'(0) = 0$.

Moreover, $(s-1)^2$ divides $p(s)$, i.e., $p(1) = p'(1) = 0$.
Indeed, $p(1) = q_{\bm{D}}(\bm{x}(1)) = 0$ by $\bm{x}(1) \in C \subseteq Z_2$ (see \eqref{eq:nno2_soc3_Z2}) and \eqref{eq:nno2_soc3_CPZ2_exposed}.
Since $\bm{e}_1 + \bm{e}_2$ and $\bm{x}(1)$ belong to the convex cone $C$, which is contained in $Z_2$, it follows from \eqref{eq:nno2_soc3_CPZ2_exposed} that
\begin{equation}
0 = \frac{1}{2}q_{\bm{D}}(\bm{e}_1 + \bm{e}_2 + \bm{x}(1)) = \bm{e}_1^\top\bm{D}\bm{x}(1) + \bm{e}_2^\top\bm{D}\bm{x}(1). \label{eq:e1pe2_D_x1}
\end{equation}
Let $i \in \{1,2\}$ and $t > 0$ be arbitrary.
By $t\bm{e}_i + \bm{x}(1) \in \bbR_+^2 \times L_+^3$ and $\bm{D}\in \COP(\bbR_+^2 \times L_+^3)$, we have
\begin{equation*}
0 \le q_{\bm{D}}(t\bm{e}_i + \bm{x}(1)) = t^2q_{\bm{D}}(\bm{e}_i) + 2t\bm{e}_i^\top\bm{D}\bm{x}(1).
\end{equation*}
Dividing this by $2t$ and letting $t\downarrow 0$, we have $\bm{e}_i^\top\bm{D}\bm{x}(1) \ge 0$.
Combining this with \eqref{eq:e1pe2_D_x1}, since $i\in \{1,2\}$ is arbitrary, we obtain
\begin{equation}
\bm{e}_1^\top\bm{D}\bm{x}(1) = \bm{e}_2^\top\bm{D}\bm{x}(1) = 0. \label{eq:ei_D_x1_0}
\end{equation}
For $s\in \bbR$, let
\begin{equation*}
\overline{\bm{x}}(s) \coloneqq \left(0,0,\frac{1+s^2}{2},\frac{1-s^2}{2},s\right)^\top \in \bbR_+^2 \times L_+^3
\end{equation*}
and $\bar{p}(s) \coloneqq q_{\bm{D}}(\overline{\bm{x}}(s))$.
The polynomial $\bar{p}$ satisfies $\bar{p}(s) \ge 0$ for all $s\in \bbR$ by $\bm{D}\in \COP(\bbR_+^2 \times L_+^3)$ and $\bar{p}(1) = 0$ by $\overline{\bm{x}}(1) = \bm{x}(1) \in Z_2$ and \eqref{eq:nno2_soc3_CPZ2_exposed}.
This implies that
\begin{equation*}
0 = \bar{p}'(1) = 2\overline{\bm{x}}'(1)^\top\bm{D}\overline{\bm{x}}(1) = 2(\bm{x}'(1) + \bm{e}_2)^\top\bm{D}\bm{x}(1) = p'(1) + 2\bm{e}_2^\top\bm{D}\bm{x}(1) = p'(1),
\end{equation*}
where the first equality follows from the optimality of $\bar{p}$ at $s = 1$, the third equality follows from $\overline{\bm{x}}(s) = \bm{x}(s) - (1-s)\bm{e}_2$, and the fifth equality follows from \eqref{eq:ei_D_x1_0}.
Therefore, $s^2(s-1)^2$ divides $p(s)$.

The degree of $p$ is at most $3$ since $\bm{x}(s) = \tfrac{s^2}{2}\bm{x}^* + (0,1-s,\tfrac{1}{2},\tfrac{1}{2},s)^\top$ and $q_{\bm{D}}(\bm{x}^*) = 0$ by $\bm{x}^* \in Z_2$ (see \eqref{eq:nno2_soc3_Z2}).
Since $s^2(s-1)^2$ divides the polynomial $p(s)$ of degree at most $3$, we conclude that $p$ is the zero polynomial.
\end{proof}

It follows from \cref{lem:p_zeropoly} that $0 = p(-1) = q_{\bm{D}}(\bm{x}(-1))$.
Combining this with $\bm{x}(-1) \in \bbR_+^2\times L_+^3$ and \eqref{eq:nno2_soc3_CPZ2_exposed} yields $\bm{x}(-1)\bm{x}(-1)^\top \in \CP(Z_2)$.
However, since $q_{\bm{D}_2}(\bm{x}(-1)) = 2 \neq 0$, \eqref{eq:nno2_soc3_CPZ2_exposed_by_D2} implies that $\bm{x}(-1)\bm{x}(-1)^\top \not\in \CP(Z_2)$, which is a contradiction.
Therefore, $\CP(Z_2)$ is a non-exposed face of $\CP(\bbR_+^2\times L_+^3)$, and thus $\CP(\bbR_+^2\times L_+^3)$ is not facially exposed.
Since $\CP(\bbR_+^2 \times L_+^n)$ has an exposed face that is linearly isomorphic to $\CP(\bbR_+^2 \times L_+^3)$, by \cref{thm:nonexposed,thm:isom_face_exposed}, we conclude that $\CP(\bbR_+^2 \times L_+^n)$ is not facially exposed.

\section{Case of rank at least $5$}\label{sec:rank5}
In this section, we show that $\CP(E_+)$ is not facially exposed for any EJA $E$ of rank at least $5$.
We achieve this by constructing one of its faces that is linearly isomorphic to a completely positive cone over a nonnegative orthant.

\begin{theorem}\label{thm:rankge5}
Let $E$ be an EJA of rank at least $5$.
Then $\CP(E_+)$ is not facially exposed.
\end{theorem}

\begin{proof}
Let $r\ge 5$ denote the rank of $E$ and $\{c_1,\dots,c_r\}$ be a Jordan frame of $E$.
Then $E' \coloneqq \setspan\{c_1,\dots,c_r\}$ is a subalgebra of $E$ that is Jordan-isomorphic to the Hadamard EJA $\bbR^r$, where a Jordan isomorphism is the linear mapping sending $c_i \in E'$ to $\bm{e}_i \in \bbR^r$ for every $i \in \{1,\dots,r\}$.
It follows from \cref{lem:subalg_expface} that $\CP(E_+')$ is an exposed face of $\CP(E_+)$.
Since $\CP(\bbR_+^r)$ is not facially exposed by $r\ge 5$ and \cref{thm:CPn_facially_exposed}, it follows from \cref{cor:CP_isom_Jordan,thm:nonexposed} that $\CP(E_+)$ is not facially exposed.
\end{proof}

\vspace{0.5cm}
\noindent
{\bf Acknowledgments}
The author is supported by JSPS Grant-in-Aid for Research Activity Start-up JP25K23344.

\vspace{0.5cm}
\noindent
{\bf Declarations}
Preliminary proof ideas for \cref{lem:CP_PSD3,lem:nno2_soc3} were developed with the assistance of ChatGPT.
The author subsequently reconstructed and independently verified the proofs included in this paper and assumes full responsibility for their correctness.


\bibliographystyle{plainurl} 
\bibliography{2026_2_preprint_ref} %

@article{BMP2016,
author = {Bai, L. and Mitchell, J.E. and Pang, J.-S.},
title = {On conic {QPCC}s, conic {QCQPs} and completely positive programs},
journal = {Math. Program.},
fjournal = {Mathematical Programming},
volume = {159},
number = {1--2},
pages = {109--136},
year = {2016},
doi = {10.1007/s10107-015-0951-9},
memo = {https://doi.org/10.1007/s10107-015-0951-9}
}

@article{BDd+2000,
author = {Bomze, I.M. and D\"{u}r, M. and de Klerk, E. and Roos, C. and Quist, A.J. and Terlaky, T.},
title = {On copositive programming and standard quadratic optimization problems},
journal = {J. Global Optim.},
fjournal = {Journal of Global Optimization},
volume = {18},
number = {4},
pages = {301--320},
year = {2000},
doi = {10.1023/A:1026583532263},
memo = {https://doi.org/10.1023/A:1026583532263 J. Glob. Optim.}
}

@article{Bur2009,
author = {Burer, S.},
title = {On the copositive representation of binary and continuous nonconvex quadratic programs},
journal = {Math. Program.},
fjournal = {Mathematical Programming},
volume = {120},
number = {2},
pages = {479--495},
year = {2009},
doi = {10.1007/s10107-008-0223-z},
memo = {https://doi.org/10.1007/s10107-008-0223-z}
}

@article{Dic2011,
author = {Dickinson, P.J.C.},
title = {Geometry of the copositive and completely positive cones},
journal = {J. Math. Anal. Appl.},
fjournal = {Journal of Mathematical Analysis and Applications},
volume = {380},
number = {1},
pages = {377--395},
year = {2011},
doi = {10.1016/j.jmaa.2011.03.005},
memo = {https://doi.org/10.1016/j.jmaa.2011.03.005}
}

@book{FK1994,
author={Faraut, J. and Kor\'{a}nyi, A.},
title={Analysis on Symmetric Cones},
edition = {},
publisher={Clarendon Press},
address = {Oxford, UK},
year={1994}
}

@book{SB2021,
author={Shaked-Monderer, N. and Berman, A.},
title={Copositive and Completely Positive Matrices},
edition = {},
publisher={World Scientific},
address = {Singapore},
year = {2021},
doi = {10.1142/11386},
memo = {https://doi.org/10.1142/11386}
}

@article{SZ2003,
author = {Sturm, J.F. and Zhang, S.},
title = {On cones of nonnegative quadratic functions},
journal = {Math. Oper. Res.},
fjournal = {Mathematics of Operations Research},
volume = {28},
number = {2},
pages = {246--267},
year = {2003},
doi = {10.1287/moor.28.2.246.14485},
memo = {https://doi.org/10.1287/moor.28.2.246.14485}
}

@incollection{Bur2012,
author={Burer, S.},
title={Copositive programming},
editor={Anjos, M.F. and Lasserre, J.B.},
booktitle={Handbook on Semidefinite, Conic and Polynomial Optimization},
pages={201--218},
publisher={Springer, New York, NY},
year={2012},
doi = {10.1007/978-1-4614-0769-0_8},
memo = {https://doi.org/10.1007/978-1-4614-0769-0_8}
}

@article{BD2012,
author = {Burer, S. and Dong, H.},
title = {Representing quadratically constrained quadratic programs as generalized copositive programs},
journal = {Oper. Res. Lett.},
fjournal = {Operations Research Letters},
volume = {40},
number = {3},
pages = {203--206},
year = {2012},
doi = {10.1016/j.orl.2012.02.001},
memo = {https://doi.org/10.1016/j.orl.2012.02.001}
}

@article{NN2024_App,
author = {Nishijima, M. and Nakata, K.},
title = {Approximation hierarchies for copositive cone over symmetric cone and their comparison},
journal = {J. Global Optim.},
fjournal = {Journal of Global Optimization},
volume = {88},
number = {4},
pages = {831--870},
year = {2024},
doi = {10.1007/s10898-023-01319-3},
memo = {https://doi.org/10.1007/s10898-023-01319-3 J. Glob. Optim.}
}

@book{Roc1970,
author={Rockafellar, R.T.},
title={Convex Analysis},
edition = {},
publisher={Princeton University Press},
address = {Princeton, NJ},
year={1970}
}

@article{Orl2021,
author = {Orlitzky, M.},
title = {Gaddum's test for symmetric cones},
journal = {J. Global Optim.},
fjournal = {Journal of Global Optimization},
volume = {79},
number = {4},
pages = {927--940},
year = {2021},
doi = {10.1007/s10898-020-00960-6},
memo = {https://doi.org/10.1007/s10898-020-00960-6 J. Glob. Optim.}
}

@article{Pat2007,
author = {Pataki, G.},
title = {On the closedness of the linear image of a closed convex cone},
journal = {Math. Oper. Res.},
fjournal = {Mathematics of Operations Research},
volume = {32},
number = {2},
pages = {395--412},
year = {2007},
doi = {10.1287/moor.1060.0242},
memo = {https://doi.org/10.1287/moor.1060.0242}
}

@article{Pat2013_On,
author = {Pataki, G.},
title = {On the connection of facially exposed and nice cones},
journal = {J. Math. Anal. Appl.},
fjournal = {Journal of Mathematical Analysis and Applications},
volume = {400},
number = {1},
pages = {211--221},
year = {2013},
doi = {10.1016/j.jmaa.2012.10.033},
memo = {https://doi.org/10.1016/j.jmaa.2012.10.033}
}

@article{RG1995,
author = {Ramana, M. and Goldman, A.J.},
title = {Some geometric results in semidefinite programming},
journal = {J. Global Optim.},
fjournal = {Journal of Global Optimization},
volume = {7},
number = {1},
pages = {33--50},
year = {1995},
doi = {10.1007/BF01100204},
memo = {https://doi.org/10.1007/BF01100204 J. Glob. Optim.}
}

@article{RT2019,
author = {Roshchina, V. and Tun\c{c}el, L.},
title = {Facially dual complete (nice) cones and lexicographic tangents},
journal = {SIAM J. Optim.},
fjournal = {SIAM Journal on Optimization},
volume = {29},
number = {3},
pages = {2363--2387},
year = {2019},
doi = {10.1137/17M1126643},
memo = {https://doi.org/10.1137/17M1126643}
}

@book{BN2001,
author={Ben-Tal, A. and Nemirovski, A.},
title={Lectures on Modern Convex Optimization: Analysis, Algorithms, and Engineering Applications},
publisher={Society for Industrial and Applied Mathematics},
address = {Philadelphia, PA},
year={2001},
doi = {10.1137/1.9780898718829},
}

@article{BW1981_Reg,
author = {Borwein, J. and Wolkowicz, H.},
title = {Regularizing the abstract convex program},
journal = {J. Math. Anal. Appl.},
fjournal = {Journal of Mathematical Analysis and Applications},
volume = {83},
number = {2},
pages = {495--530},
year = {1981},
doi = {10.1016/0022-247X(81)90138-4},
memo = {https://doi.org/10.1016/0022-247X(81)90138-4},
}

@article{LRS2022,
author = {Louren\c{c}o, B.F. and Roshchina, V. and Saunderson, J.},
title = {Amenable cones are particularly nice},
journal = {SIAM J. Optim.},
fjournal = {SIAM Journal on Optimization},
volume = {32},
number = {3},
pages = {2347--2375},
year = {2022},
doi = {10.1137/20M138466X},
memo = {https://doi.org/10.1137/20M138466X}
}

@article{Zha2018,
author = {Zhang, Q.},
title = {Completely positive cones: are they facially exposed?},
journal = {Linear Algebra Appl.},
fjournal = {Linear Algebra and its Applications},
volume = {558},
number = {},
pages = {195--204},
year = {2018},
doi = {10.1016/j.laa.2018.08.028},
memo = {https://doi.org/10.1016/j.laa.2018.08.028}
}

@article{NN2024_Gen,
author = {Nishijima, M. and Nakata, K.},
title = {Generalizations of doubly nonnegative cones and their comparison},
journal = {J. Oper. Res. Soc. Japan},
fjournal = {Journal of the Operations Research Society of Japan},
volume = {67},
number = {3},
pages = {84--109},
year = {2024},
doi = {10.15807/jorsj.67.84},
memo = {https://doi.org/10.15807/jorsj.67.84 J. Oper. Res. Soc. Jpn.}
}

@article{NL2025,
author = {Nishijima, M. and Louren\c{c}o, B.F.},
title = {Non-facial exposedness of copositive cones over symmetric cones},
journal = {J. Math. Anal. Appl.},
fjournal = {Journal of Mathematical Analysis and Applications},
volume = {545},
number = {2},
pages = {129166},
year = {2025},
doi = {10.1016/j.jmaa.2024.129166},
memo = {https://doi.org/10.1016/j.jmaa.2024.129166 pagetotal = {17},}
}

@article{Kos2025,
author = {Kostyukova, O.I.},
title = {Non-exposed polyhedral faces of the completely positive cone},
journal = {Linear Multilinear Algebra},
fjournal = {Linear and Multilinear Algebra},
volume = {73},
number = {2},
pages = {368--395},
year = {2025},
doi = {10.1080/03081087.2024.2346313}
}

@article{BW1981_Fac,
author = {Borwein, J.M. and Wolkowicz, H.},
title = {Facial reduction for a cone-convex programming problem},
journal = {J. Aust. Math. Soc.},
fjournal = {Journal of the Australian Mathematical Society},
volume = {30},
number = {3},
pages = {369--380},
year = {1981},
doi = {10.1017/S1446788700017250},
memo = {https://doi.org/10.1017/S1446788700017250},
}

@book{HJ2013,
author={Horn, R.A. and Johnson, C.R.},
title={Matrix Analysis},
edition = {Second},
publisher={Cambridge University Press},
address = {New York, NY},
year={2013},
doi = {10.1017/CBO9781139020411},
memo = {https://doi.org/10.1017/CBO9781139020411}
}

@article{MM1962,
author={Maxfield, J.E. and Minc, H.},
title={On the Matrix Equation {$X'X = A$}},
journal={Proc. Edinburgh Math. Soc.},
fjournal={Proceedings of the Edinburgh Mathematical Society},
volume={13},
number = {2},
pages={125--129},
year={1962},
doi={10.1017/S0013091500014681},
memo = {https://doi.org/10.1017/S0013091500014681 Proc. Edinb. Math. Soc.},
}

@article{NL2026,
author = {Nishijima, M. and Louren\c{c}o, B.F.},
title = {Facial structure of copositive and completely positive cones over a second-order cone},
journal = {Linear Algebra Appl.},
fjournal = {Linear Algebra and its Applications},
volume = {743},
number = {},
pages = {211--249},
year = {2026},
doi = {10.1016/j.laa.2026.04.019},
memo = {https://doi.org/10.1016/j.laa.2026.04.019}
}

@article{GST2004,
author = {Gowda, M.S. and Sznajder, R. and Tao, J.},
title = {Some {{\bf P}}-properties for linear transformations on {E}uclidean {J}ordan algebras},
journal = {Linear Algebra Appl.},
fjournal = {Linear Algebra and its Applications},
volume = {393},
number = {},
pages = {203--232},
year = {2004},
doi = {10.1016/j.laa.2004.03.028},
memo = {https://doi.org/10.1016/j.laa.2004.03.028}
}

@article{Nis2026,
author = {Nishijima, M.},
title = {Copositive and completely positive cones over symmetric cones of rank at least 5},
journal = {J. Optim. Theory Appl.},
fjournal = {Journal of Optimization Theory and Application},
volume = {210},
number = {3},
year = {2026},
pages = {55},
doi = {10.1007/s10957-026-03077-0},
memo = {https://doi.org/10.1007/s10957-026-03077-0 pagetotal = {36},}
}

@article{Hil1888,
author = {Hilbert, D.},
title = {\"{U}ber die {D}arstellung definiter {F}ormen als {S}umme von {F}ormenquadraten},
journal = {Math. Ann.},
fjournal = {Mathematische Annalen},
volume = {32},
number = {3},
pages = {342--350},
year = {1888},
doi = {10.1007/BF01443605},
memo = {https://doi.org/10.1007/BF01443605}
}

@article{Zha2020,
author = {Zhang, Q.},
title = {Faces of the $5\times 5$ completely positive cone},
journal = {Linear Multilinear Algebra},
fjournal = {Linear and Multilinear Algebra},
volume = {68},
number = {12},
pages = {2523--2540},
year = {2020},
doi = {10.1080/03081087.2019.1586827},
memo = {https://doi.org/10.1080/03081087.2019.1586827}
}

@article{Kos2024,
author = {Kostyukova, O.I.},
title = {Non-exposed faces of the cone of completely positive matrices},
journal = {Proceedings of the Institute of Mathematics of the National Academy of Sciences of Belarus},
volume = {32},
number = {2},
pages = {56--68},
year = {2024},
memo = {journal = Proc. Inst. Math. Natl. Acad. Sci. Belarus}
}

@article{Fle1985,
author = {Fletcher, R.},
title = {Semi-definite matrix constraints in optimization},
journal = {SIAM J. Control Optim.},
fjournal = {SIAM Journal on Control and Optimization},
volume = {23},
number = {4},
pages = {493--513},
year = {1985},
doi = {10.1137/0323032},
memo = {https://doi.org/10.1137/0323032}
}

@article{Tau1968,
author = {Taussky, O.},
title = {The factorization of the adjugate of a finite matrix},
journal = {Linear Algebra Appl.},
fjournal = {Linear Algebra and its Applications},
volume = {1},
number = {1},
pages = {39--41},
year = {1968},
doi = {10.1016/0024-3795(68)90046-3},
memo = {https://doi.org/10.1016/0024-3795(68)90046-3}
}

@incollection{NN2016,
author={N\'{e}meth, A.B. and N\'{e}meth, S.Z.},
title={Lattice-like Subsets of {E}uclidean {J}ordan Algebras},
editor={Rassias, T. and Pardalos, P.},
booktitle={Essays in Mathematics and its Applications},
pages={159--179},
publisher={Springer, Cham, Switzerland},
year={2016},
doi = {10.1007/978-3-319-31338-2_8},
memo = {https://doi.org/10.1007/978-3-319-31338-2_8}
}

@Misc{HX20XX,
author = {Huang, L. and Xie, L.},
title = {A finite-termination algorithm for testing copositivity over the positive semidefinite cone},
eprint = {2601.06648},
year = {2026},
eprintclass = {math.OC},
memo = {https://doi.org/10.48550/arXiv.2601.06648}
}

@article{FGN+2024,
author = {Ferreira, O.P. and Gao, Y. and N\'{e}meth, S.Z. and  Rig\'{o}, P.R.},
title = {Gradient projection method on the sphere, complementarity problems and copositivity},
journal = {J. Global Optim.},
fjournal = {Journal of Global Optimization},
volume = {90},
number = {1},
pages = {1--25},
year = {2024},
doi = {10.1007/s10898-024-01390-4},
memo = {https://doi.org/10.1007/s10898-024-01390-4 J. Glob. Optim.}
}

@article{Orl2025_JorAut,
author = {Orlitzky, M.},
title = {Jordan automorphisms and derivatives of symmetric cones},
journal = {Linear Algebra Appl.},
fjournal = {Linear Algebra and its Applications},
volume = {721},
number = {},
pages = {26--46},
year = {2025},
doi = {10.1016/j.laa.2024.04.024},
memo = {https://doi.org/10.1016/j.laa.2024.04.024}
}

@article{Ros2014,
author = {Roshchina, V.},
title = {Facially exposed cones are not always nice},
journal = {SIAM J. Optim.},
fjournal = {SIAM Journal on Optimization},
volume = {24},
number = {1},
pages = {257--268},
year = {2014},
doi = {10.1137/130922069},
memo = {https://doi.org/10.1137/130922069}
}
\end{document}